%% file: main.tex
\documentclass[11pt]{article}

\usepackage{girth}
\usepackage{scalerel}

\newcommand{\TODO}[1]{}

\newcommand{\Lining}{\diamond}
\newcommand{\UN}{\mathrm{UN}}
\newcommand{\ER}{\bbG}
\newcommand{\Reg}{\bbR}
\newcommand{\nice}{\text{nice}}
\newcommand{\dist}{\mathrm{dist}}
\newcommand{\depth}{\mathsf{depth}}
\newcommand{\Leaves}{\mathsf{Leaves}}
\newcommand{\exc}{\operatorname{exc}}

\begin{document}
\title{Explicit two-sided unique-neighbor expanders}

\author{Jun-Ting Hsieh\thanks{Carnegie Mellon University.  \texttt{juntingh@cs.cmu.edu}.  Supported by NSF CAREER Award \#2047933.}
\and Theo McKenzie\thanks{Stanford University.  \texttt{theom@stanford.edu}. Supported by NSF GRFP Grant DGE-1752814 and NSF Grant DMS-2212881.}
\and Sidhanth Mohanty\thanks{MIT. \texttt{sidhanth@csail.mit.edu}.  Much of this work was conducted while the author was a PhD student at UC Berkeley.}
\and Pedro Paredes\thanks{Princeton University. \texttt{pparedes@cs.princeton.edu}.}}

\maketitle

\begin{abstract}
    We study the problem of constructing explicit sparse graphs that exhibit strong vertex expansion. Our main result is the first two-sided construction of imbalanced unique-neighbor expanders, meaning bipartite graphs where small sets contained in both the left and right bipartitions exhibit unique-neighbor expansion, along with algebraic properties relevant to constructing quantum codes.

    Our constructions are obtained from instantiations of the \emph{tripartite line product} of a large tripartite spectral expander and a sufficiently good constant-sized unique-neighbor expander, a new graph product we defined that generalizes the \emph{line product} in the work of Alon and Capalbo \cite{AC02} and the \emph{routed product} in the work of Asherov and Dinur \cite{AD23}.
    To analyze the vertex expansion of graphs arising from the tripartite line product, we develop a sharp characterization of subgraphs that can arise in bipartite spectral expanders, generalizing results of Kahale \cite{Kah95}, which may be of independent interest.
    
    By picking appropriate graphs to apply our product to, we give a strongly explicit construction of an infinite family of $(d_1,d_2)$-biregular graphs $(G_n)_{n\ge 1}$ (for large enough $d_1$ and $d_2$) where all sets $S$ with fewer than a small constant fraction of vertices have $\Omega(d_1\cdot |S|)$ unique-neighbors (assuming $d_1 \leq d_2$).
    Additionally, we can also guarantee that subsets of vertices of size up to $\exp(\Omega(\sqrt{\log |V(G_n)|}))$ expand \emph{losslessly}.

\end{abstract}

\thispagestyle{empty}
\setcounter{page}{0}
\newpage

\thispagestyle{empty}
\setcounter{page}{0}
\tableofcontents
\thispagestyle{empty}
\setcounter{page}{0}
\newpage

\input{intro}
\input{prelim}
\input{line}
\input{construction}
\input{spectral-radius}
\input{subgraph-density}
\input{un-lossless}
\input{lossless}

\newpage 
\bibliographystyle{alpha}
\bibliography{main}

\newpage

\appendix
\input{moore}
\input{gadget}

\input{sandwich}

\end{document}

%% file: intro.tex
\section{Introduction}
A bipartite graph $G$ is a \emph{one-sided unique-neighbor expander} if every small subset of its left vertices has many \emph{unique-neighbors}, where a unique-neighbor of a set $S$ is a vertex $v$ with exactly one edge to $S$.
Classically, there is a wealth of applications of one-sided unique-neighbor expanders to error-correcting codes~\cite{SS96,DSW06,BV09}, high-dimensional geometry~\cite{BGI+08,GLR10,Kar11,GMM22}, and routing~\cite{ALM96},
as well as several explicit constructions \cite{AC02,CRVW02,AD23,Gol23,Coh}.

A recent work of Lin \& Hsieh \cite{LH22b} established a connection between quantum error-correcting codes and \emph{two-sided unique-neighbor expanders}, which are graphs where every small subset of both the left and right vertices has many unique-neighbors.
In particular, they showed that good \emph{quantum low-density parity check (LDPC) codes} with efficient decoding algorithms can be obtained from two-sided \emph{lossless} expanders satisfying certain algebraic properties, with the additional advantage of being simpler to analyze than earlier constructions of good quantum codes \cite{PK22,LZ22}.
Here, lossless expanders are graphs achieving the quantitatively strongest form of unique-neighbor expansion possible.
A random biregular graph is a two-sided lossless expander with high probability, but no explicit constructions are known; all explicit constructions of one-sided unique-neighbor expanders are not known to satisfy two-sided expansion.

The main contribution of this work is to give explicit constructions of infinite families of two-sided unique-neighbor expanders.
We now delve into our results and provide context.





\subsection{Our results}

We give a formal description of the graphs we would like to construct, motivated by constructing quantum codes, and then describe our contributions.
\begin{definition}[Two-sided (algebraic) unique-neighbor expander]
    We say a $(d_1,d_2)$-biregular graph $Z$ with left and right vertex sets $L$ and $R$ respectively is a $\gamma$-\emph{two-sided unique-neighbor expander} if there is a constant $\delta$ depending on $\gamma, d_1, d_2$ such that:
    \begin{enumerate}
        \item Every subset $S\subseteq L$ with $|S| < \delta |L|$ has at least $\gamma \cdot d_1|S|$ neighbors in $R$.
        \item Every subset $S\subseteq R$ with $|S| < \delta |R|$ has at least $\gamma \cdot d_2|S|$ neighbors in $L$.
    \end{enumerate}
    We say $Z$ is a $\gamma$-\emph{two-sided algebraic unique-neighbor expander} if additionally: there is a group $\Gamma$ of size $\Omega(|L|+|R|)$ that acts on $L$ and $R$ such that $gv = v$ iff $g$ is the identity element of $\Gamma$, and $\{gu,gv\}$ is an edge iff $\{u,v\}$ is an edge in $Z$.
\end{definition}
The work of Lin \& Hsieh \cite{LH22b} proves that for $\gamma > 1/2$, the existence of $\gamma$-two-sided algebraic unique-neighbor expanders with arbitrary aspect ratio implies the existence of linear-time decodable quantum LDPC codes.

We give an explicit construction of $\gamma$-two-sided algebraic unique-neighbor expanders for small constant $\gamma$.
\begin{theorem}[Two-sided algebraic unique-neighbor expanders.]   \label{thm:main}
    For every $\beta\in(0,1/2]$ there is a constant $\gamma > 0$ such that for all large enough $d_1,d_2$ with $1\ge\frac{d_1}{d_2} \ge \frac{\beta}{1-\beta}$, there is an explicit infinite family of $(2d_1, 2d_2)$-biregular graphs $(Z_n)_{n\ge 1}$ where every $Z_n$ is a $\gamma$-two-sided algebraic unique-neighbor expander.
\end{theorem}
We refer the reader to \Cref{thm:main-thm} for a formal statement.

\begin{remark}
    \Cref{thm:main} gives the first construction of two-sided unique-neighbor expanders where the left and right side have unequal sizes.\footnote{Throughout the paper we will assume the left side is larger, i.e.\ $d_1 \leq d_2$.}
    It also gives the only construction besides the one-sided lossless expander constructions of \cite{CRVW02,Gol23,Coh} where the number of unique-neighbors of a set $S$ can be made arbitrarily larger than $|S|$.
    We give a detailed comparison to prior work in \Cref{tab:previous-constructions}.
\end{remark}

\begin{remark}
    Our techniques also straightforwardly generalize to constructing families of bounded degree $k$-partite unique-neighbor expanders for any distribution $(\beta_1,\dots,\beta_k)$ of vertices across partitions.
\end{remark}

We also give constructions of two-sided unique-neighbor expanders where we can additionally guarantee that small enough sets expand losslessly, at the expense of the algebraic property.
\begin{theorem}[Two-sided unique-neighbor expanders with small-set lossless expansion; see \Cref{thm:2-sided-un-sse}] 
\torestate{ \label{thm:main-2}
    For every $\beta\in(0,1/2]$ and $\eps > 0$, there are constants $\gamma > 0$ and $K$ such that for all large enough $d_1,d_2$ with $1\ge\frac{d_1}{d_2} \ge \frac{\beta}{1-\beta}$ that are multiples of $K$, there is an explicit infinite family of $(d_1, d_2)$-biregular graphs $(Z_n)_{n\ge 1}$ where:
    \begin{enumerate}
        \item $Z_n$ is a $\gamma$-two-sided unique-neighbor expander,
        \item every $S\subseteq L(Z_n)$ with $|S| \le \exp(\Omega(\sqrt{\log |V(Z_n)|}))$ has $(1-\eps)\cdot d_1\cdot |S|$ unique-neighbors,
        \item every $S\subseteq R(Z_n)$ with $|S| \le \exp(\Omega(\sqrt{\log |V(Z_n)|}))$ has $(1-\eps)\cdot d_2\cdot |S|$ unique-neighbors.
    \end{enumerate}
}
\end{theorem}

    At a high level, all of our constructions involve taking a certain product of a large ``base graph'' with a constant-sized ``gadget graph''.
    In \Cref{thm:main,thm:main-2}, the unique-neighbor expansion comes from strong spectral expansion properties of the base graph; see \Cref{sec:technical-overview} for an overview.
    The algebraic property is also inherited from the base graph satisfying the same algebraic property.
    These two properties can be simultaneously achieved by choosing the base graph as Ramanujan Cayley graphs \cite{LPS88,Mar88,Mor94}.

\begin{remark}[Bicycle-free Ramanujan graph construction]
    In \Cref{thm:main-2}, the small-set lossless expansion property comes from the base graph consisting of bipartite spectral expanders with no short bicycles (\Cref{def:excess}): no bicycles of length-$g$ roughly translates to lossless expansion for sets of size $\exp(g)$; see \Cref{thm:2-sided-un-sse} for a formal statement.
    We believe that the biregular Ramanujan graph construction of \cite{BFGRKMW15} should have no bicycles of length-$\Omega(\log n)$ and also endow a group action, but we do not prove it in this work.
    We instead use constructions from the works of \cite{MOP20,OW20}, which have no bicycles of length $\Omega(\sqrt{\log n})$ but no group action.
\end{remark}

\begin{remark}[One-sided lossless expanders] \label{rem:recover-lossless}
    As explained in more detail in the technical overview (\Cref{sec:technical-overview}), our graph product generalizes the \emph{routed product} defined in \cite{AC02,AD23,Gol23}.
    In particular, by instantiating the product with slightly different parameters, we are able to prove one-sided lossless expansion with essentially the same proof as \Cref{thm:main-2}, recovering the result of Golowich~\cite{Gol23}.
    The analysis is carried out in \Cref{sec:one-sided-lossless}.
\end{remark}

\parhead{New results in spectral graph theory.}
In service of proving \Cref{thm:main-2}, we prove two results that we believe to be independently interesting in spectral graph theory:
\begin{enumerate}[(i)]
    \item we give a sharp characterization of what subgraphs can arise in bipartite spectral expanders, generalizing results of Kahale \cite{Kah95} and Asherov \& Dinur \cite{AD23},
    \item we give a refinement to the well-known \emph{irregular Moore bound} of \cite{AHL02} on the tradeoff between girth and edge density in a graph.
\end{enumerate}

In particular, we show that for any small induced subgraph of a near-Ramanujan biregular graph, the spectral radius of its \emph{non-backtracking matrix} (see \Cref{sec:nb-matrix}) must be bounded.

\begin{theorem}[See \Cref{thm:biregular-subgraph}]
\label{thm:subgraph-spectral-radius-main}
    Let $\eps \in (0,0.1)$, and let $3 \leq c \leq d$ be integers.
    Let $G = (L \cup R, E)$ be a $(c,d)$-biregular graph and $S \subseteq L \cup R$ such that $|S| \leq d^{-1/\eps} |L\cup R|$.
    Then,
    \begin{equation*}
        \rho(B_{G[S]}) \leq \frac{1}{2} \Paren{\sqrt{\wt{\lambda}^2 - (\sqrt{c-1}+\sqrt{d-1})^2} + \sqrt{\wt{\lambda}^2 - (\sqrt{c-1}-\sqrt{d-1})^2} } \mcom
    \end{equation*}
    where $\wt{\lambda} = \max(\lambda_2(A_G), \sqrt{c-1}+\sqrt{d-1}) \cdot (1+O(\eps))$.
\end{theorem}

\begin{remark} \label{rem:spectral-radius-intuition}
    For the sake of intuition, we inspect what \Cref{thm:subgraph-spectral-radius-main} tells us in the special case where $G$ is a biregular near-Ramanujan graph.
    When we plug in $\wt{\lambda} = \parens*{\sqrt{c-1} + \sqrt{d-1}}\cdot (1+\eps)$, we obtain:
    \[
        \rho(B_{G[S]}) \le \parens*{ (c-1) (d-1) }^{1/4} + \delta(\eps)
    \]
    where $\delta(\eps)\to 0$ as $\eps \to 0$.
\end{remark}

When $c = d$, Kahale's result for $d$-regular graphs (e.g., Theorem 3 of \cite{Kah95}) also has the form $\frac{1}{2} \parens{\wt{\lambda} + \sqrt{\wt{\lambda}^2 - 4(d-1)}}$.
The above expression thus generalizes Kahale's result to biregular graphs.

\begin{remark}[Sharpness of \Cref{thm:subgraph-spectral-radius-main}]
    One can adapt the techniques of \cite{MM21} to prove that for \emph{any} graph $H$ on $o(n)$ vertices where
    \[
        \rho(B_H) \le \frac{1}{2}\parens*{ \sqrt{\lambda^2 - (\sqrt{c-1}+\sqrt{d-1})^2} + \sqrt{\lambda^2 - (\sqrt{c-1}-\sqrt{d-1})^2} } \mcom
    \]
    there is a graph $G$ that contains $H$ as a subgraph, and $\lambda_2(A_G) \le \lambda\cdot(1+o(1))$.
\end{remark}

As a consequence, we obtain the following result which answers a question raised by \cite{AD23}.

\begin{theorem}[Subgraph density in (near-)Ramanujan graphs; see \Cref{thm:expansion}]
\label{thm:subgraph-density-main}
    Let $3 \leq c \leq d$ be integers, and let $G = (L \cup R, E)$ be a $(c,d)$-biregular graph such that $\lambda_2(G) \leq (\sqrt{c-1} + \sqrt{d-1})(1 + \gamma/d)$.
    Then, there exists $\delta = \delta(\eps,c,d) > 0$ such that for every $S_1 \subseteq L$ and $S_2 \subseteq R$ with $|S_1| + |S_2| \leq \delta |L \cup R|$,
    the left and right average degrees $d_1 = \frac{|E(S_1,S_2)|}{|S_1|}$ and $d_2 = \frac{|E(S_1,S_2)|}{|S_2|}$ in the induced subgraph $G[S_1 \cup S_2]$ must satisfy
    \begin{equation*}
        (d_1-1) (d_2-1) \leq \sqrt{(c-1)(d-1)} \cdot (1+O(\eps+\sqrt{\gamma})) \mper
    \end{equation*}
\end{theorem}

Our refinement to the Moore bound involves using the spectral radius of the non-backtracking matrix of the graph instead of the degree, and yields the existence of \emph{bicycles} --- pairs of short cycles that are close in the graph.
\begin{theorem}[Generalized Moore bound; see \Cref{thm:moore-bound}] \label{thm:moore-bound-main}
    Let $G$ be a graph on $n$ vertices and let $\rho = \lambda_1(B_G)$ where $B_G$ is the non-backtracking matrix of $G$.
    Assuming $\rho > 1$, $G$ must contain a cycle of length at most $(2+o_n(1))\log_{\rho} n$ and $G$ must contain a bicycle of length at most $(3+o_n(1))\log_{\rho} n$.
\end{theorem}

\begin{remark} \label{rem:generalized-moore-bound}
    For a graph $G$ with average degree $d$, $\rho(B_G)$ is at least $d-1$.
    Therefore, \Cref{thm:moore-bound-main} is stronger than the girth guarantee of $2\log_{d-1} n$ from the classical irregular Moore bound of \cite{AHL02} in some cases, in particular for some graphs arising in the proof of \Cref{thm:main-2}.
    A simple example where this yields tighter bounds is a $(d,2)$-biregular graph.
    When $d\gg 2$, the average degree is $\approx 4$ and the classical Moore bound yield a cycle of length $2\log_3 n$.
    Nevertheless, the generalized Moore bound tells us that there is a cycle of length $\approx 4\log_{d-1} n$.
\end{remark}

\subsection{Context and related work}
\parhead{Spectral expansion vs.\ unique-neighbor expansion.}
In contrast to unique-neighbor expanders, we have a rich set of spectral and edge expander constructions.
A key conceptual difficulty in unique-neighbor expansion is the lack of an ``analytic handle'' for it.
Several other graph properties required in applications of expander graphs, such as high conductance on cuts, low density of small subgraphs, and rapid mixing of random walks, have an excellent surrogate in the second eigenvalue of the normalized adjacency matrix, which is a highly tractable quantity.

It is natural to wonder if any form of unique-neighbor expansion can be deduced from the eigenvalues of a graph.
However, the connection between the unique-neighbor expansion in a graph and its spectral properties is tenuous at best.
Kahale \cite{Kah95} proved that in $d$-regular graphs with optimal spectral expansion, small sets have vertex expansion at least $d/2$, i.e., for a sufficiently small constant $\eps$, any set $S$ with at most $\eps n$ vertices has roughly at least $|S|\cdot d/2$ distinct neighbors.
Observe that once the vertex expansion of a set exceeds $d/2$, it begins to be forced to have unique-neighbors.
Strikingly, Kahale also showed the $d/2$ bound is tight for spectral expanders, which makes them fall short at the cusp of the unique-neighbor expansion threshold; indeed, it was proved in \cite{KK22} that certain algebraic bipartite Ramanujan graphs contain sublinear-sized sets with \emph{zero} unique-neighbors (see also \cite{Kah95,MM21} for examples of near-Ramanujan graphs exhibiting a similar property).

\parhead{Quantum codes.}
Resilience to errors is essential for constructing quantum computers \cite{Kit03}, which makes quantum error correction fundamental for quantum computing.
One approach to this problem is in the form of \emph{quantum LDPC codes} (qLDPC codes).
Recently, a flurry of work culminated in the construction of qLDPC codes with constant rate and distance \cite{PK22,LZ22}, which was also coupled with the construction of $c^3$-\emph{locally testable codes}~\cite{DEL+22,PK22}.
At a high level, these codes are constructed by composing a structured spectral expander, a \emph{square Cayley complex}, along with a structured inner code, a \emph{robustly testable tensor code}.
The analysis is complicated by the stringent requirements on the inner code, and poses a barrier for generalizing these constructions.
Indeed, it is unclear how to generalize the square Cayley complex and the inner code to construct \emph{quantum locally-testable codes} (qLTCs).

More recently, Lin and Hsieh~\cite{LH22b} constructed good qLDPC codes with \emph{linear time decoders} assuming the existence of two-sided algebraic lossless expander graphs.
Their construction does not require an inner code, and as a byproduct, yields a simpler analysis and is plausibly easier to generalize to other applications such as qLTCs.\footnote{More recently qLDPCs with linear time decoders have been constructed~\cite{DHLV22,GPT22,LZ23}, but they still make use of an inner code.}
However, two-sided algebraic lossless expander graphs are not known to exist, and obtaining them is one of the primary motivations for the goals of this paper.

The \emph{chain complexes} arising in the qLDPC constructions have also been fruitful for other problems in theoretical computer science --- constructing explicit integrality gaps for the Sum-of-Squares semidefinite programming hierarchy for the $k$-XOR problem \cite{HL22} \& the resolution of the quantum NLTS conjecture \cite{ABN23}.

The previously known integrality gaps for $k$-XOR came from \emph{random} instances \cite{Gri01,Sch08}.
Building on the work of Dinur, Filmus, Harsha \& Tulsiani \cite{DFHT21}, Hopkins \& Lin \cite{HL22} constructed \emph{explicit} families of 3-XOR instances that are hard for the Sum-of-Squares (SoS) hierarchy of semidefinite programming relaxations (previously known lower bounds are \emph{random} instances).
Specifically, they illustrated $k$-XOR instances which are highly unsatisfiable but even $\Omega(n)$ levels of SoS fail to refute them (i.e., perfect completeness).


\parhead{Previous constructions.}
The first constructions of unique-neighbor expanders appeared in the work of Alon and Capalbo \cite{AC02}.
One of their constructions, which we extend in this paper, takes the line product of a large Ramanujan graph with the $8$-vertex $3$-regular graph obtained by the union of the octagon and edges connecting diametrically opposite vertices.

Another construction in the same work gives one-sided unique-neighbor expanders of aspect ratio $22/21$, and was extended in a recent work of Asherov and Dinur \cite{AD23} to obtain one-sided unique-neighbor expanders of aspect ratio $\alpha$ for all $\alpha \ge 1$ where every small set on the left side has at least $1$ unique-neighbor.
The construction takes a graph product called the \emph{routed product} of a large biregular Ramanujan graph with a constant-sized random graph.
(See the recent work of Kopparty, Ron-Zewi \& Saraf \cite{kopparty2023simple} for a simplified analysis with weaker ingredients.)


The work of Capalbo, Reingold, Vadhan \& Wigderson \cite{CRVW02} constructs one-sided lossless expanders of arbitrarily large degree and arbitrary aspect ratio.
Their construction relies on a generalization of the \emph{zig-zag product} of \cite{RVW00} applied to various randomness conductors to construct lossless conductors, analyzed by tracking entropy, which then translates to lossless expanders.
More recently, a simpler construction and analysis was given by Golowich \cite{Gol23} based on the routed product (see also \cite{Coh} for a similar construction, and see \Cref{rem:nonexpanding} for a discussion on where the routed product constructions fall short of achieving two-sided expansion).

Finally, motivated by randomness extractors, the works \cite{TUZ07,GUV09} construct one-sided lossless expanders where the left side is polynomially larger than the right.

\newcommand{\specialcell}[2][c]{%
  \begin{tabular}[#1]{@{}c@{}}#2\end{tabular}}
\newcommand{\dingcheck}{{\color{Emerald}\small \ding{51}}}
\newcommand{\dingboldcheck}{{\small \ding{52}}}
\newcommand{\dingx}{{\color{Mahogany} \small \ding{55}}}
\newcommand{\dingboldx}{{\small \ding{56}}}

\begin{table}[ht]
\caption{Comparison of our \Cref{thm:main} with prior work.}
{\centering
\begin{tabular}{lccccc}
\hline
\label{tab:previous-constructions}
 {\textbf{Construction?}} &
 { \textbf{Which \emph{d}?}${}^{\dagger}$} &
 \specialcell{{\textbf{\# unique-neighbors}} \\ { \textbf{of $S$} }}  &
 { \textbf{2-sided?}} &
 \specialcell{{ \textbf{Explicit?}}} &
  \specialcell{{ \textbf{Aspect} } \\ {\textbf{ratio} ${}^{\ddagger}$} }
 \\ \hhline{======}
 { Random graphs}  & {any $d$}           & {$(d-\eps)\cdot |S|$}                       & \dingcheck            & \dingx     & {any}     \\
 { \cite{AC02}}    & {$\{3,4,6\}$}       & {$\Omega(|S|)$}                  & \dingcheck${}^{\ast}$                  & \dingcheck & 1${}^{\ast}$ \\
 { \cite{AC02}}    & {$\subseteq [25]$}          & {$\Omega(|S|)$}                  & \dingx                & \dingcheck & { $22/21$ } \\
 $\begin{array}{c}\text{\cite{CRVW02}} \\ \text{\cite{Gol23,Coh}}\end{array}$  & {large enough $d$}  & {$(d-o(d))\cdot|S|$}          & \dingx                & \dingcheck & {any} \\
 { \cite{Bec16}}   & {$6$}               & {$\Omega(|S|)$ }                 & \dingcheck${}^{\ast}$                  & \dingcheck & 1$^{\ast}$ \\
 { \cite{AD23,kopparty2023simple}} & {large enough $d$}  & {at least $1$}                      & \dingx                & \dingcheck & {any}   \\  \hline
 { \textbf{this paper}} & {large enough $d$}  & {$\Omega(d|S|)$}       & \dingcheck            & \dingcheck & {any} \\ \hline
\end{tabular}}
\vspace{8pt}

${}^{\ast}$Non-bipartite construction that can be made bipartite by passing to the double cover.

${}^{\dagger}$$d$ here refers to the degree of the left vertex set.

${}^{\ddagger}$``Aspect ratio'' refers to the ratio between the sizes of the left and right vertex sets.
\end{table}

\parhead{Applications of unique-neighbor expanders.}
Unique-neighbor expanders have several applications in theoretical computer science.
In coding theory, it was shown in \cite{DSW06,BV09} that unique-neighbor expander codes~\cite{Tan81} are ``weakly smooth'', hence when tensored with a code with constant relative distance, they give \emph{robustly testable codes}.
In high-dimensional geometry, unique-neighbor expanders were used in~\cite{GLR10,Kar11} to construct \emph{$\ell_p$-spread} subspaces as well as in~\cite{BGI+08,GMM22} to construct matrices with the \emph{$\ell_p$-restricted isometry property} (RIP).

Unique-neighbor expanders were also used in designing \emph{non-blocking networks}~\cite{ALM96}: given a set of input and output terminals, the network graph is connected such that no matter which input-output pairs are connected previously, there is a path between any unused input-output pair using unused vertices.


\subsection{Technical overview}
\label{sec:technical-overview}

\parhead{Line product.}
Our construction of two-sided algebraic unique-neighbor expanders, featured in \Cref{thm:main}, is based on the \emph{line product} between a large \emph{base graph} and a small \emph{gadget graph}.
Let $G$ be a $D$-regular graph on $n$ vertices and $H$ be a $d$-regular graph on $D$ vertices. The line product $G \Lining H$ is a graph on the \emph{edges} of $G$ where for each vertex $v\in G$ we place a copy of $H$ on the set of edges incident to $v$.
See \Cref{def:line-product} for a formal definition and \Cref{fig:line-example} for an example.
This graph product was also used in the works of \cite{AC02,Bec16}.

\begin{figure}[ht!]
    \centering
    \includegraphics[width=0.8\textwidth]{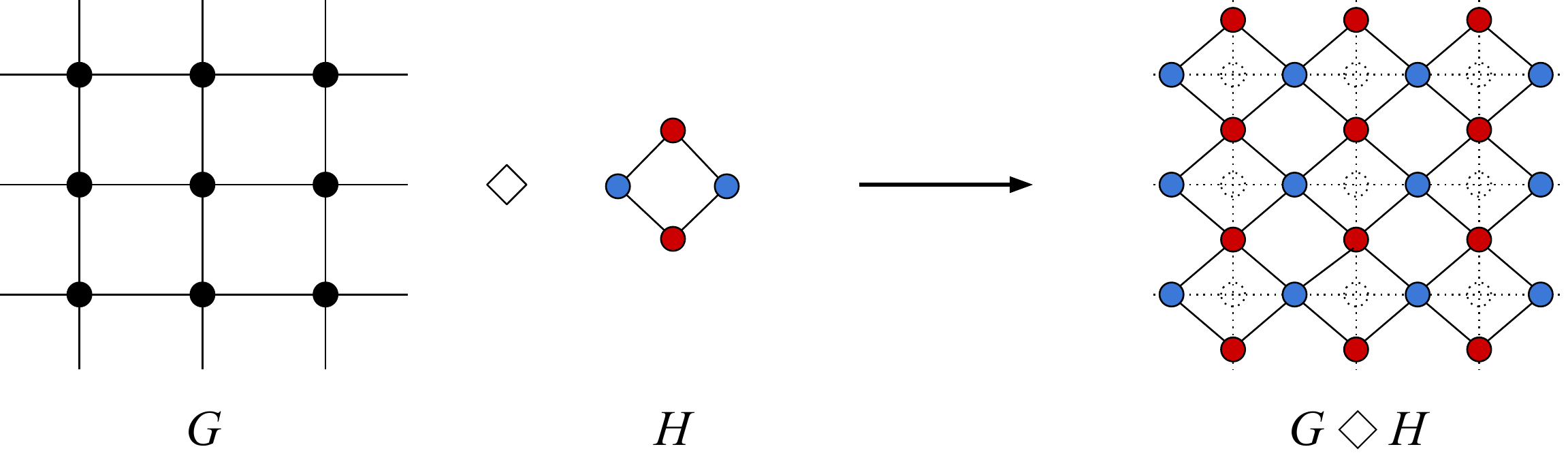}
    \caption{An example of the line product.}
    \label{fig:line-example}
\end{figure}

Observe that $G \Lining H$ has $nD/2$ vertices and is $2d$-regular.
Note also that the \emph{line graph} of $G$ (where two edges are connected if they share a vertex) is exactly the line product between $G$ and the $D$-clique, hence the name.

The key lemma (\Cref{lem:line-analysis}) is that if $G$ is a small-set (edge) expander and $H$ is a good unique-neighbor expander, then $G \Lining H$ is a unique-neighbor expander as well.
For the base graph $G$, we simply use the explicit Ramanujan graph construction~\cite{LPS88,Mor94}.
For the gadget $H$, we show that a \emph{random} biregular graph is a good unique-neighbor expander with high probability (\Cref{lem:gadget}).
Then, since $D$ is a constant, we can find such a graph by brute force.

\begin{remark}[On importance of Ramanujan base graphs]
    We require an $O(\sqrt{D})$ bound on the average degree of small subgraphs in a $D$-regular expander, which is proved using the fact that the second eigenvalue of a $D$-regular Ramanujan graph is $O(\sqrt{D})$.
    Typically, applications of expanders only need a second eigenvalue bound of $o(D)$, so we find it noteworthy that the analysis of our construction seems to require being within a constant factor of the Ramanujan bound.
\end{remark}

\parhead{Tripartite line product.}
Our constructions with stronger vertex expansion guarantees for small sets, featured in \Cref{thm:main-2}, are based on a suitably generalized version of the line product, which we call the \emph{tripartite line product}.
The first ingredient is a large tripartite \emph{base graph} $G$ on vertex set $L\cup M\cup R$, where $L, R, M$ denote the left, middle, right partitions respectively, and there are bipartite graphs between $L$, $M$ and between $M$, $R$.
The second ingredient is a small constant-sized bipartite \emph{gadget graph} $H$, which is chosen to be an excellent unique-neighbor expander --- as before, we can brute-force search to find $H$ that has expansion as good as a random graph.

The tripartite line product $G\Lining H$ is a bipartite graph on $L\cup R$ whose edges are obtained by placing a copy of $H$ between the left neighbors of $v$ and the right neighbors of $v$ for each $v\in M$.
See \Cref{def:tripartite-line} for a formal definition and \Cref{fig:tripartite-line-example} for an example.

\begin{figure}[ht!]
    \centering
    \includegraphics[width=0.8\textwidth]{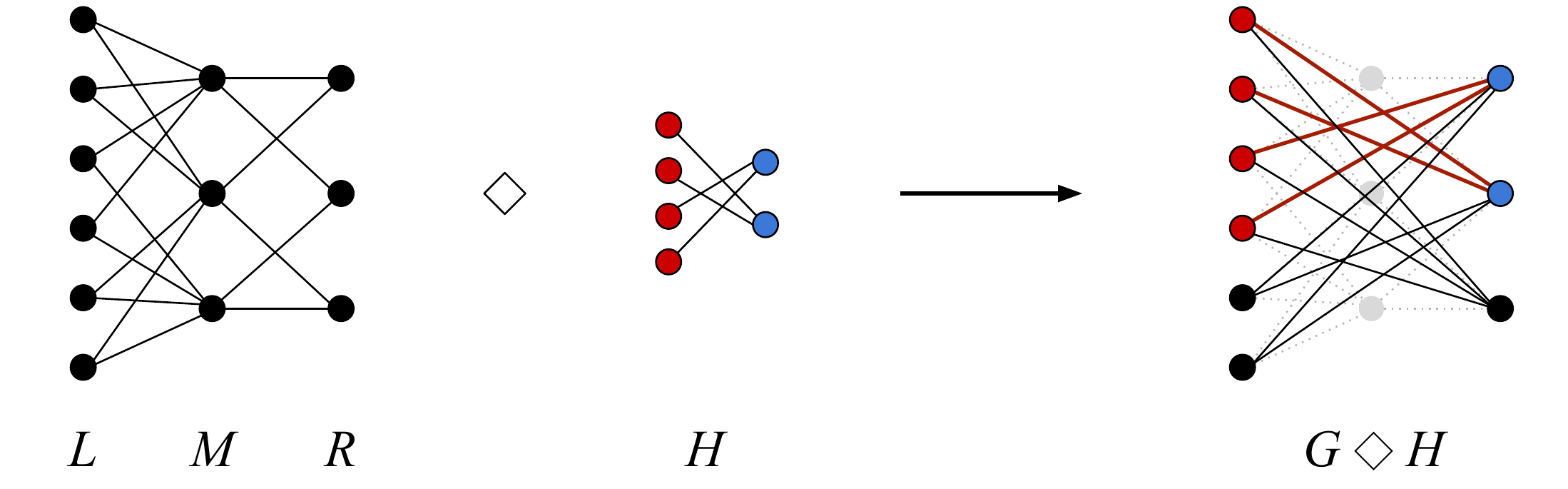}
    \caption{An example of the tripartite line product. The gadget placed on the first vertex in $M$ is highlighted in $G \Lining H$.}
    \label{fig:tripartite-line-example}
\end{figure}

To prove \Cref{thm:main-2}, we construct the base graph by choosing a $(K_1,D_1)$-biregular spectral expander between $L$ and $M$, and a $(D_2,K_2)$-biregular spectral expander between $M$ and $R$ for large enough and suitably chosen parameters $K_1,D_1,K_2,D_2$.
The gadget $H$ is chosen to be a $(\wt{d}_1,\wt{d}_2)$-biregular graph on vertex set $[D_1] \times [D_2]$, where $d_1 = K_1 \wt{d}_1$ and $d_2 = K_2 \wt{d}_2$.

\begin{remark}[Generalizing the line product and routed product]
    The \emph{line product} and \emph{routed product}, which feature in \cite{AC02,AD23,Gol23}, arise from instantiating the tripartite line product with appropriate base graphs.
    The line product can be obtained by choosing a $(2,D_1)$-biregular graph between $L$ and $M$, and a $(D_2,2)$-biregular graph between $M$ and $R$ in the base graph.
    The routed product arises by choosing a $(K,D_1)$-biregular graph between $L$ and $M$, and a $(D_2,1)$-biregular graph between $M$ and $R$ in the base graph.
\end{remark}

\begin{remark}[One-sided vs.\ two-sided expanders]  \label{rem:nonexpanding}
    A key difference between our work and previous constructions that only achieve one-sided expansion is in the choice of the graph between $M$ and $R$.
    \cite{AD23,Gol23} choose the graph between $M$ and $R$ to be a $(D_2,1)$-biregular graphs, equivalently a disjoint collection of stars centered at the vertices in $M$.
    This results in very small sets on the right with no unique-neighbors: for example, for any $v\in M$, consider the set of all its right neighbors.
\end{remark}

\parhead{Overview of the analysis of the tripartite line product.}
Let $Z = G \Lining H$, and let $G^{(1)}$ and $G^{(2)}$ be the bipartite graphs between $L, M$ and $M, R$ in $G$ respectively.
We choose the gadget $H$ to be a $(\wt{d}_1, \wt{d}_2)$-biregular graph on $D = D_1+D_2$ vertices such that $\wt{d}_1 \gg \frac{1}{\eps}\sqrt{K_2 D_2}$.
For a subset $S \subseteq L$, let $U = N_G(S) \subseteq M$ (neighbors of $S$ in the base graph).
Our analysis for the expansion of $S$ roughly follows two steps:
\begin{enumerate}[(1)]
    \item We partition $U$ according to the number of edges going to $S$ --- $U_{\ell}$ (``low $S$-degree'') and $U_h$ (``high $S$-degree'').
    We then show that if we partition $U$ according to a suitable threshold, then most edges leaving $S$ go to $U_\ell$.
    \label{step:left-expansion}

    \item Since each vertex in $U_{\ell}$ has small $S$-degree, in the local gadget graph it has ``large'' unique-neighbor expansion (here we rely on the expansion profile of the gadget; see \Cref{lem:gadget}).
    Then, we show that most of the unique-neighbors in the gadgets are also unique-neighbors of $S$ in $Z$.
    \label{step:right-expansion}
\end{enumerate}

Both steps rely on \Cref{thm:subgraph-density-main}.
For step \ref{step:left-expansion}, we apply \Cref{thm:subgraph-density-main} on the induced subgraph $G^{(1)}[S \cup U_h]$ (since $G^{(1)}$ is near-Ramanujan).
If we choose the threshold to be $\approx \sqrt{D}$, then since the right average degree of $G^{(1)}[S \cup U_h]$ is $\gtrsim \sqrt{D}$, the left average degree is $\lesssim 1+ \frac{\sqrt{K_1D_1}}{\sqrt{D}} \ll K_1$, thus most edges from $S$ go to $U_{\ell}$ instead of $U_h$.

For step \ref{step:right-expansion}, let $\wt{T} \subseteq R$ be the union of the unique-neighbors within each gadget.
By the expansion of the gadgets on vertices in $U_{\ell}$, we have a lower bound on $|\wt{T}|$.
However, a unique-neighbor in one gadget may also have edges from other gadgets, in which case it is \emph{not} a unique-neighbor of $S$ in $Z$.
To resolve this issue, we analyze the induced subgraph $G^{(2)}[U \cup \wt{T}]$ and show that a large fraction of $\wt{T}$ are unique-neighbors of $U$ in $G^{(2)}$, thus must also be unique-neighbors of $S$ in $Z$.
This is done by observing that the left average degree of $G^{(2)}[U \cup \wt{T}]$ must be $\gtrsim \wt{d}_1$.
Thus, the right average degree is $\lesssim 1 + \frac{\sqrt{K_2 D_2}}{\wt{d}_1} \leq 1 + \eps$ since we choose $\wt{d}_1 \gg \frac{1}{\eps}\sqrt{K_2 D_2}$.

One might attempt to tweak the parameters of the construction to obtain \emph{two-sided} lossless expansion.
However, this fails because in step \ref{step:left-expansion} we need the threshold to be large enough such that $S$ has lossless expansion in $G^{(1)}$, but then it is not possible to set the parameters of the gadget such that (i) each vertex in $U_{\ell}$ expands losslessly, and (ii) $\wt{d}_1 \gg \frac{1}{\eps} \sqrt{K_2 D_2}$ (and $\wt{d}_2 \gg \frac{1}{\eps} \sqrt{K_1 D_1}$) for the analysis in step \ref{step:right-expansion}.
See \Cref{sec:tripartite-line-product} (the proof of \Cref{thm:2-sided-un-sse}) for details.

\parhead{Lossless expansion of small sets.}
For small subsets $S \subseteq L$, we directly show that $S$ expands losslessly into $U$ under the assumption that $G^{(1)}$ has no short bicycles.
Specifically, in \Cref{lem:small-set-expansion} we prove that if a graph has no bicycle of length $g$, then all sufficiently small subsets (in particular, of size at most $\exp(O(g))$) expand losslessly.
To prove this, we first show that if a degree-$K$ set $S$ has expansion less than $(1-\eps)K$, then we can lower bound the spectral radius of the non-backtracking matrix of the graph induced on $S\cup N(S)$ by $C(\eps) \coloneqq \sqrt{\eps(K-1)}$.
Now, by the generalized Moore bound (\Cref{thm:moore-bound-main}), there is a bicycle of size $O(\log_{C(\eps)} |S|)$ in $G$.
Since $G$ has no bicycle of length-$g$, it lower bounds $|S|$ via the inequality $g\le O(\log_{C(\eps)} |S|)$.
This tells us that small sets in $G^{(1)}$ exhibit lossless expansion.
To establish that small sets in $Z$ are losslessly expanding, we follow the same strategy as before for step \ref{step:right-expansion}: since most vertices in $U$ has only 1 edge to $S$, they expand by a factor of $\wt{d}_1$ (from the gadget), and we use \Cref{thm:subgraph-density-main} to show that $T = N_Z(S)$ are mostly unique-neighbors, proving the small set expansion result in \Cref{thm:main-2}.

\parhead{One-sided lossless expanders.}
As mentioned in \Cref{rem:recover-lossless}, we are able to use the tripartite line product to construct \emph{one-sided} lossless expanders, recovering the result of Louis~\cite{Gol23}.
Our proof is almost the same as \Cref{thm:main-2}, but with $K_1 \gg K_2 = 1$ (hence $G^{(2)}$ is just a collection of stars).
In step \ref{step:left-expansion}, we choose a smaller $S$-degree threshold $\approx \eps \sqrt{D}$ and large enough $K_1$ to ensure that $1-\eps$ fraction of edges from $S$ go to $U_{\ell}$ (via \Cref{thm:subgraph-density-main}).
Then, with the smaller threshold, each vertex in $U_{\ell}$ expands losslessly, i.e., by a factor of $(1-\eps)\wt{d}_1$.
Since all gadgets have disjoint right vertices, there is no collision between gadgets, which finishes the proof.

\parhead{Subgraphs in near-Ramanujan bipartite graphs.}
We now give an overview of the proofs of \Cref{thm:subgraph-spectral-radius-main,thm:subgraph-density-main}.
The main ingredient is the \emph{Bethe Hessian} of a graph $G$ defined as $H_G(t) \coloneqq (D_G - \Id)t^2 - A_G t + \Id$.
Specifically, for an induced subgraph $G[S]$, we identify $\alpha > 0$ such that $H_{G[S]}(t) \succ 0$ for $t\in[0,\alpha]$ (\Cref{thm:biregular-subgraph}), which then implies \Cref{thm:subgraph-spectral-radius-main} via the well-known Ihara--Bass formula (\Cref{fact:ihara-bass}).

To prove $H_{G[S]}(t) \succ 0$ for $t\in[0,\alpha]$, we show that $f^\top H_{G[S]}(t) f > 0$ for any $f: G[S] \to \R$.
To establish $f^\top H_{G[S]}(t) f > 0$ for the claimed values of $t$, we need to relate $f^\top H_{G[S]}(t) f$ to the spectrum of the entire graph $G$.
To this end, we consider the \emph{regular tree extension} $T$ of $G[S]$ (\Cref{def:regular-tree-extension}), which is obtained by attaching trees (of depth $\ell$) to each vertex in $S$ such that the resulting graph is $(c,d)$-biregular except for the leaves.
Intuitively, the tree extension serves as a proxy of the $\ell$-step neighborhood of $S$ in $G$.
We then define the appropriate function extension $f_t$ of $f$ (\Cref{def:function-extension}) on the tree extension with ``sufficient decay'' down the tree (where the decay results from an upper bound on $t$).
For appropriately bounded $t$, we can show $f^{\top} H_{G[S]}(t) f$ is approximately equal to $f_t^{\top} H_{G}(t) f_t$, which can be controlled via the spectrum of $G$.

\parhead{\erdos--\renyi vs.\ random regular graphs.}
One technical subtlety is that the edges of a random regular graph are (slightly) \emph{correlated}, which makes it difficult to directly analyze its unique-neighbor expansion.
On the other hand, the analysis is straightforward for \erdos--\renyi graphs since the edges are drawn independently (\Cref{lem:erdos-renyi-gadget}).

Thus, we give a slight generalization of an embedding theorem given in \cite{FK16} between \erdos--\renyi graphs and random regular graphs (\Cref{lem:stoc-dom}) which allows us to extend the analysis to random regular graphs.

\subsection{Open questions}

\noindent {\bf Quantum codes from unique-neighbor expanders.} Lin \& Hsieh \cite{LH22b} construct quantum codes assuming the existence of two-sided algebraic lossless expanders, and their current proof requires the unique-neighbor expansion of sets of vertices with degree-$d$ to exceed $d/2$.

In contrast, in the setting of classical LDPC codes, if all subsets of vertices of size at most $\Delta$ have even a single unique-neighbor, the resulting code is guaranteed to have distance at least $\Delta$, albeit without a clear decoding algorithm.

\begin{question}
    Does the construction of \cite{LH22b} yield a good quantum code when instantiated with a $\gamma$-two-sided algebraic unique-neighbor expander for small $\gamma > 0$?
\end{question}

\parhead{Algebraic two-sided lossless expanders from random graphs.}
Two-sided lossless expanders with relevant algebraic properties are not known to exist, even using randomness.
Random bipartite graphs exhibit two-sided lossless expansion (see, for example, \cite[Theorem 4.16]{HLW18}), but do not admit any nontrivial group actions.

A natural approach is to use an algebraic random graph, such as a random Cayley graph, and study its vertex expansion properties.
It is also possible to achieve lossless expansion using the tripartite line product if the small-set edge expansion of the base graph is far better than that guaranteed by spectral expansion.
A concrete question in this direction that could result in a lot of progress is the following.
\begin{question}[Beyond spectral certificates]
    Let $\Gamma$ be a group and let $\bS$ be a set of $D$ generators chosen independently and uniformly from $\Gamma$.
    Let $\bG \coloneqq \mathrm{Cay}(\Gamma,\bS)$ be the Cayley graph given by the generator set $\bS$.
    Is it true that with high probability over the choice of $\bS$: for all subsets of vertices $T\subseteq\Gamma$ where $|T|\le|\Gamma|/\mathrm{poly}(D)$, the number of edges inside $T$ is at most $0.1\sqrt{D}|T|$?
\end{question}
\begin{remark}
    Spectral expansion can at best guarantee that the number of edges inside a set $T$ is at most $\sqrt{D}|T|$.
    A resolution to the above question would necessarily use other properties of the random graph and group, beyond merely the magnitude of its eigenvalues.
\end{remark}

\subsection{Organization}

In \Cref{sec:prelims}, we give some technical preliminaries.
In \Cref{sec:line-product}, we define the line product and prove its unique-neighbor expansion assuming good expansion properties of the base and gadget graphs.
In \Cref{sec:un-expanders}, we prove \Cref{thm:main} by instantiating the line product with a suitably chosen base graph and gadget graph.

In \Cref{sec:nb-specrad}, we prove \Cref{thm:subgraph-spectral-radius-main}.
In \Cref{sec:subgraph-density}, we use \Cref{thm:subgraph-spectral-radius-main} to prove \Cref{thm:subgraph-density-main}.
In \Cref{sec:ss-une}, we define the tripartite line product, and instantiate it with an appropriately chosen base graph and gadget graph to prove \Cref{thm:main-2} via \Cref{thm:subgraph-density-main} and \Cref{thm:moore-bound-main}.
In \Cref{sec:one-sided-lossless}, we show how we can recover Golowich's \cite{Gol23} construction and analysis of one-sided lossless expanders as an instantiation of the tripartite line product and an application of \Cref{thm:subgraph-density-main}.

In \Cref{app:exact-moore-bound}, we prove the sharpened Moore bound \Cref{thm:moore-bound-main}.
In \Cref{sec:gadget,sec:sandwich}, we analyze the expansion profile of the gadget graph.

%% file: prelim.tex
\section{Preliminaries}
\label{sec:prelims}

\paragraph{Notation.}  Given a graph $G$, we use $V(G)$ to denote its set of vertices, $E(G)$ to denote its set of edges.
If $G$ is bipartite, we use $L(G)$ and $R(G)$ to denote its left and right vertex sets respectively.
We write $\deg_G(v)$ to denote the degree of vertex $v$ in $G$ (we will omit the subscript $G$ if clear from context).
We say that a bipartite graph $G$ is $(d_1,d_2)$-biregular if $\deg_G(v) = d_1$ for all $v\in L(G)$ and $\deg_G(v) = d_2$ for all $v\in R(G)$.

For any subset of edges $F$, we use $F(v) \subseteq F$ to denote the set of edges in $F$ incident to $v$, hence $|F(v)| = \deg_F(v)$.
For a set of vertices $S$, we use $e_F(S)$ to denote the number of edges in $F$ with both endpoints in $S$, and $e_F(S,T)$ to denote the number of $(u,v) \in S\times T$ such that $\{u,v\} \in F$ (for disjoint sets $S$ and $T$).
We will omit the subscript if $F = E(G)$.
We denote the eigenvalues of the normalized adjacency matrix of $G$ by $1 = \lambda_1(G)\ge\dots\ge\lambda_n(G) \geq -1$.
We say $G$ is a $\lambda$-spectral expander if $\lambda_2(G) \le \lambda$.\footnote{This is in contrast to most scenarios where one requires both $\lambda_2(G)$ as well as $-\lambda_n(G)$ to be bounded by $\lambda$.}

\paragraph{Random bipartite graph models.}
Throughout the paper, we will write random variables in \textbf{boldface}.
Fix $n_1, n_2, m, d_1, d_2 \in \N$ such that $n_1 d_1 = n_2 d_2$.
We denote $K_{n_1,n_2}$ as the complete bipartite graph with left/right vertex sets $L = [n_1]$ and $R = [n_2]$.
We use $\bG \sim \ER_{n_1, n_2, m}$ to denote a random graph sampled from the uniform distribution over (simple) bipartite graphs on $L = [n_1]$, $R = [n_2]$ with exactly $m$ edges.
With slight abuse of notation, we use $\bH \sim \ER_{n_1,n_2,p}$ for $p \in (0,1)$ to denote a random graph such that each potential edge is included with probability $p$.
Similarly, we use $\bR \sim \Reg_{n_1, n_2, d_1, d_2}$ to denote a graph from the uniform distribution over $(d_1,d_2)$-biregular bipartite graphs on $L = [n_1]$, $R = [n_2]$.

\subsection{Graph expansion}
It is a standard fact that small subgraphs in spectral expanders have bounded number of edges.
The following is a special case of the Expander Mixing Lemma~\cite{AC88}, and we include a short proof for completeness.
\begin{lemma}   \label{lem:EML-same-set}
    Let $G$ be a $D$-regular $\lambda$-spectral expander.
    Then for any set $S\subseteq V(G)$, where $|S| = \eps |V(G)|$:
    \[
        e(S) \le D |S|\cdot \frac{\lambda + \eps}{2} \mper
    \]
\end{lemma}
\begin{proof}
    Let $n = |V(G)|$, $A$ be the (unnormalized) adjacency matrix of $G$, and $\one_S\in \{0,1\}^n$ be the indicator vector of $S$.
    We can decompose $\one_S$ as $\one_S = \frac{|S|}{n}\one + u$ where $u \perp \one$ and $\|u\|_2^2 = |S|(1-\frac{|S|}{n}) \leq |S|$.
    Then,
    $2e(S) = \one_S^{\top} A \one_S \le \frac{D|S|^2}{n} + \lambda_2(G) \cdot D \|u\|_2^2 \le D |S|(\eps + \lambda)$.
\end{proof}

Within graphs of low ``hereditary'' average degree, a significant fraction of edges are incident to low-degree vertices.
\begin{lemma}   \label{lem:most-low}
    For any $\gamma>0$, let $F$ be a graph such that for all $S\subseteq V(F)$, $2e(S) \le \gamma|S|$.  Write $V(F) = F_{\ell}\sqcup F_h$ where $F_{\ell}$ comprises all vertices $v$ such that $\deg(v) \le 2\gamma$ and $F_h$ to denote the remaining vertices.
    Then:
    \[
        2e(F) \le 3\sum_{v\in F_{\ell}} \deg(v) \mper
    \]
\end{lemma}
\begin{proof}
    On one hand, by assumption we have:
    \[
        2e(F_h) \le \gamma |F_h| \mper
    \]
    On the other hand,
    \[
        2e(F_h) + e(F_h, F_{\ell}) = \sum_{v\in F_h} \deg(v) \ge 2\gamma|F_h| \ge 4e(F_h) \mper
    \]
    Consequently $e(F_h, F_{\ell}) \ge 2e(F_h)$.
    Since $F_h$ and $F_{\ell}$ are disjoint:
    \begin{align*}
        2e(F) = 2e(F_h) + e(F_h, F_\ell) + e(F_{\ell}, F_h) + 2 e(F_\ell) \mper
    \end{align*}
    This gives us:
    \(
        2e(F) \le 3e(F_h, F_{\ell}) + 2e(F_{\ell}) \le 3 \sum_{v\in F_{\ell}} \deg(v) \mper
    \)
\end{proof}


\subsection{Non-backtracking matrix}
\label{sec:nb-matrix}

\paragraph{Notation.}
Given an undirected graph $G = (V, E)$ with $|V| = n$ vertices and $|E|= m$ edges, we denote $A_G \in \{0,1\}^{n\times n}$ to be its adjacency matrix, $D_G\in \R^{n\times n}$ to be its diagonal degree matrix, and finally $B_G \in \{0,1\}^{2m\times 2m}$ to be its non-backtracking matrix (introduced by \cite{Hashimoto89}) defined as follows: for \emph{directed} edges $(u_1,v_1), (u_2,v_2)$ in the graph,
\begin{align*}
    B_G[(u_1,v_1), (u_2,v_2)] = \1(v_1 = u_2) \cdot \1(u_1 \neq v_2) \mper
\end{align*}
Note that the non-backtracking matrix is not symmetric.
Let $\lambda_1, \dots, \lambda_{2m} \in \C$ be the eigenvalues of $B_G$ ordered such that $|\lambda_1| \geq |\lambda_2| \geq \cdots \geq |\lambda_{2m}|$. The Perron-Frobenius theorem implies that $\lambda_1$ is real and non-negative.
We denote $\rho(B_G) = \lambda_1$ to be the spectral radius of $B_G$.

A crucial identity we will need is the Ihara-Bass formula~\cite{Ihara66,Bass92} which gives a translation between the eigenvalues of the adjacency matrix and the eigenvalues
of the non-backtracking matrix:

\begin{fact}[Ihara-Bass formula] \label{fact:ihara-bass}
    For any graph $G$ with $n$ vertices and $m$ edges,
    the following identity on univariate polynomials is true:
    \[
        \det\parens*{\Id - B_Gt} = \det\parens*{ H_G(t) } \cdot (1-t^2)^{m-n}
    \]
    where $H_G(t)\coloneqq(D_G-\Id)t^2 - A_G t + \Id$ is the \emph{Bethe Hessian} of $G$.
\end{fact}

The Ihara-Bass formula gives a direct relationship between the spectral radius of $B_G$ and the positive definiteness of $H_G(t)$.
The following is classic (e.g. \cite[Proof of Theorem 5.1]{FM17}), though we include a proof for completeness.

\begin{lemma} \label{lem:spectral-radius}
    Let $G$ be a graph and $0 < \alpha < 1$.
    Then, the spectral radius $\rho(B_G) \leq \frac{1}{\alpha}$ if and only if $H_G(t) \succ 0$ for all $t \in [0, \alpha)$.
    As a result, if $H_G(\frac{1}{\rho})$ has a non-positive eigenvalue for some $\rho > 0$, then $\rho(B_G) \geq \rho$.
\end{lemma}
\begin{proof}
    First observe that $H_G(0) = \Id \succ 0$.
    Since $H_G(t)$ is symmetric, the eigenvalues of $H_G(t)$ are real and move continuously on the real line as $t$ increases from $0$.
    Note also that by the Perron-Frobenius theorem, $\rho(B_G) = \lambda_1(B_G) \geq 0$.

    Suppose for contradiction that $\rho(B_G) \leq \frac{1}{\alpha}$ but $H_G(t)$ has a non-positive eigenvalue for some $t \in [0, \alpha)$.
    Due to $H_G(0) = \Id$ and continuity of the eigenvalues, there must be a $t^* \in (0, t]$ such that $H_G(t^*)$ has a zero eigenvalue, meaning $\det(H_G(t^*)) = 0$.
    By \Cref{fact:ihara-bass}, this means that $\det(\Id - B_G t^*) = 0$, i.e., $B_G$ has an eigenvalue $\frac{1}{t^*} \geq \frac{1}{t} > \frac{1}{\alpha}$.
    This contradicts that $\rho(B_G) \leq \frac{1}{\alpha}$.

    On the other hand, if $H_G(t) \succ 0$ for all $t\in [0, \alpha)$, then by \Cref{fact:ihara-bass} $\det(\Id- B_G t) > 0$ for all $t\in[0,\alpha)$.
    Since $\det(\Id - B_G / \lambda_1) = 0$, it follows that $\frac{1}{\lambda_1} \ge \alpha$, i.e., $\lambda_1 \leq \frac{1}{\alpha}$.

    Finally, suppose $H_G(\frac{1}{\rho}) \not\succ 0$ for some $\rho > 0$, then setting $\alpha = \frac{1}{\rho-\eps} > \frac{1}{\rho}$ for any $\eps \to 0^+$, we have that $\rho(B_G) > \rho-\eps$.
    This implies that $\rho(B_G) \geq \rho$. 
\end{proof}

%% file: line.tex
\section{The line product of graphs}
\label{sec:line-product}

Our construction is based on taking the \emph{line product} of a suitably chosen spectral expander and unique-neighbor expander.
See \Cref{fig:line-example} for an example.

\begin{definition}[Line product]  \label{def:line-product}
    Let $G$ be a $D$-regular graph on vertex set $[n]$, and let $H$ be a graph on vertex set $[D]$.
    For each $v\in V(G)$ and $i\in[D]$, let $e_v^i$ denote the $i$-th incident edge to $v$.
    The \emph{line product} $G\Lining H$ is the graph on vertex set $E(G)$ and edges obtained by placing a copy of $H$ on $E(v)$ for each $v\in V(G)$, such that $\{e_v^i, e_v^j\}$ is an edge in $H(v)$ if and only if $\{i,j\}$ is an edge in $H$.
    
    For convenience, we denote $H(v)$ to be the subgraph of $G \Lining H$ given by the copy of $H$ associated with $v$.
\end{definition}

\begin{definition}
    Given a graph $H$, we denote $\UN_H(S)$ to be the set of unique-neighbors of $S$.
    The \emph{unique-neighbor expansion profile} of a graph $H$, denoted $P_H$, is defined:
    \[
        P_H(t) \coloneqq \min_{S\subseteq V(H): |S| \le t} \frac{|\UN_H(S)|}{|S|}
        \quad \textnormal{for $t \geq 1$}.
    \]
    When $H$ is a bipartite graph, we use $P_H^{(L)}(t)$ and $P_H^{(R)}(t)$ to denote the analogous quantity where the minimum is taken only over subsets of the left and right respectively.
\end{definition}

\begin{lemma}[Expansion profile of the line product]   \label{lem:line-analysis}
    Let $\gamma > 0$ and $\eps \in (0,1)$.
    Suppose
    \begin{enumerate}
        \item $G$ is a $D$-regular graph such that
        $2 e(S) \leq \gamma |S|$ for all $S\subset V(G)$ with $|S| \leq \eps |V(G)|$, and
        \item $H$ is a graph on $[D]$ such that $P_H(t) \ge 12\gamma/t$ for $1 \le t \le 2\gamma$.
    \end{enumerate}
    Then for $Z\coloneqq G\Lining H$, $P_Z\left({\frac{\eps}{D}}\cdot |V(Z)| \right) \ge \frac{P_H(2\gamma)}{3}$.
\end{lemma}
\begin{proof}
    Let $T\subseteq V(Z)$, viewed as a collection of edges in $G$, be such that $|T| \le {\frac{\eps}{D}}\cdot |V(Z)| = \frac{\eps}{2} \cdot |V(G)|$ (since $|V(Z)| = |E(G)| = \frac{D}{2}|V(G)|$).
    Let $S \subseteq V(G)$ be the set of vertices of $G$ touched by $T$. 
    Note that $|S| \le 2|T| \le \eps |V(G)|$.
    Recall that we denote $T(v) \subseteq T$ to be the set of edges in $T$ incident to $v$, and we have $\deg_T(v) = |T(v)| \leq \deg_{G[S]}(v)$.
    Then,
    \begin{align*}
        \abs*{\UN_Z(T)} = \sum_{v\in S} \abs*{\UN_{H(v)}(T(v))} - \sum_{v\in S} \sum_{\substack{v'\in S \\ v \neq v'}} \abs*{\UN_{H(v)}(T(v)) \cap \UN_{H(v')}(T(v'))} \mper
    \end{align*}
    Each summand in the first term of the right-hand side can be lower bounded using $P_H$ (see below).
    A vertex $\{v_1,v_2\} \in V(Z)$ (an edge in $G$) is contained in exactly two subgraphs $H(v_1)$ and $H(v_2)$, so it can only be counted twice in the second term, which means we can bound the whole sum by $2e(S)$.

    Since $|S| \leq \eps |V(G)|$, by the assumption on $G$, both the induced subgraph $G[S]$ and $T$ (viewed as a subgraph of $G[S]$) satisfy the assumption of \Cref{lem:most-low}, i.e., all $S' \subseteq S$ satisfies $2e_T(S') \leq 2e(S') \leq \gamma |S'|$.
    Thus, define $S_1 \coloneqq \{v\in S: \deg_T(v) \leq 2\gamma\}$ and $S_2 \coloneqq \{v\in S: \deg_{G[S]}(v) \leq 2\gamma\} \subseteq S_1$.
    Applying \Cref{lem:most-low} to both $T$ and $G[S]$, we get that $2|T| = 2e_T(S) \leq 3\sum_{v\in S_1} \deg_T(v)$ and
    $2e(S) \leq 3\sum_{v\in S_2} \deg_{G[S]}(v) \leq 6\gamma |S_2|\leq 6\gamma |S_1|$, hence
    \begin{align*}
        \abs*{\UN_Z(T)} &\ge \sum_{v\in S} |T(v)| \cdot P_H\parens*{|T(v)|} - 2e(S) \\
        &\ge \sum_{v\in S} |T(v)| \cdot P_H\parens*{|T(v)|} - 6\gamma |S_1| \\
        &= \sum_{v\in S_1} |T(v)| \cdot \parens*{P_H\parens*{|T(v)|} - \frac{6\gamma}{|T(v)|} } + \sum_{v\in S\setminus S_1} |T(v)| \cdot P_H\parens*{|T(v)|} \\
        &\ge \frac{1}{2} \sum_{v\in S_1} |T(v)| \cdot P_H\parens*{|T(v)|} \mcom
    \end{align*}
    where the last inequality is by $|T(v)| = \deg_T(v) \leq 2\gamma$ for all $v\in S_1$ and the assumption that $P_H(t) \geq 12\gamma/t$ for all $t\leq 2\gamma$.
    Then since $P_H(t)$ is monotonically decreasing with $t$ and $\sum_{v\in S_1} |T(v)| \geq \frac{2}{3}|T|$,
    \[
        \frac{\abs*{\UN_Z(T)}}{|T|} \ge \frac{1}{2|T|}\sum_{v\in S_1} |T(v)| \cdot P_H(2\gamma) \ge \frac{P_H(2\gamma)}{3} \mper \qedhere
    \]
\end{proof}

%% file: construction.tex
\section{Algebraic unique-neighbor expanders}
\label{sec:un-expanders}

We use a Ramanujan graph equipped with symmetries bestowed by its Cayley graph structure as our base spectral expander.
\begin{fact}[{Ramanujan graph construction \cite{LPS88,Mor94}}]    \label{fact:ramanujan-cayley}
    For every $D = p^r + 1$ where $p$ is prime and $r\in \N$, there is an infinite family of groups $(\Gamma_n)_{n\in\N}$ and a collection of generators $A\subseteq\Gamma_n$ closed under inversion where $|A| = D$ such that the Cayley graph $G \coloneqq \mathrm{Cay}(\Gamma_n, A)$ is a $D$-regular Ramanujan graph, i.e., it is a $\frac{2\sqrt{D-1}}{D}$-spectral expander.
\end{fact}

For arbitrary $D$, by deleting a few edges, we can get $D$-regular Cayley graphs that satisfy the expanding condition in \Cref{lem:line-analysis} with similar parameters as Ramanujan graphs.

\begin{lemma}[Expanding Cayley graphs of every degree] \label{lem:ramanujan-deleted}
    For every $D\in \N$, $D \geq 3$, there is an infinite family of groups $(\Gamma_n)_{n\in \N}$ and a collection of generators $A\subseteq \Gamma_n$ closed under inversion where $|A| = D$ such that the Cayley graph $G \coloneqq \mathrm{Cay}(\Gamma_n,A)$ is a $D$-regular graph such that
    \begin{equation*}
        2 e(S) \leq 3D |S| \Paren{\frac{2\sqrt{D-1}}{D} + \eps},
        \quad \forall S\subseteq V(G),\ |S| = \eps |V(G)|.
    \end{equation*}
\end{lemma}
\begin{proof}
    For odd $D$, there exists an $r\in \N$ such that $D-1 \leq 2^r \leq 2(D-1)$, thus we set $D' = 2^r+1 \leq 2D$.
    For even $D$, there exists an $r\in \N$ such that $D-1 \leq 3^r \leq 3(D-1)$, thus we set $D' = 3^r+1 \leq 3D$.
    We construct the $D'$-regular graph $G' = \mathrm{Cay}(\Gamma_n,A')$ via \Cref{fact:ramanujan-cayley} such that $\lambda_2(G') \leq \frac{2\sqrt{D'-1}}{D'} \leq \frac{2\sqrt{D-1}}{D}$ (since $\frac{2\sqrt{x-1}}{x}$ is decreasing with $x$ for $x \geq 2$).

    Since $D$ and $D'$ have the same parity, we can remove pairs of generators $g \neq g^{-1}$ from $A'$ until there are $D$ elements left (if at some point only self-inverse elements remain, then we start removing them one at a time).
    Let $A$ be the remaining generators with $|A| = D$.
    By construction, $G \coloneqq \mathrm{Cay}(\Gamma_n,A)$ is $D$-regular.

    Now, we upper bound $e(S)$.
    Deleting edges can only decrease $e(S)$, so by \Cref{lem:EML-same-set},
    \begin{equation*}
        2 e(S) \leq D' |S| (\lambda_2(G') + \eps)
        \leq 3D |S| \Paren{\frac{2\sqrt{D-1}}{D} + \eps}.
    \end{equation*}
    This completes the proof.
\end{proof}


For the gadget, we will need a unique-neighbor expander with strong quantitative guarantees.
\begin{lemma}   \torestate{\label{lem:gadget}
    Let $\beta\in(0,1/2]$, $\theta > 0$, and $\tau > 0$ be constants.
    For integers $d_1, d_2, D_1, D_2$ and $D\coloneqq D_1 + D_2$ satisfying
    \begin{enumerate}
        \item $\frac{d_1}{d_2} = \frac{D_2}{D_1}$,
        \item $1 \ge \frac{d_1}{d_2} \ge \frac{\beta}{1-\beta}$,
        \item $\theta \sqrt{D}/2 \le d_1 + d_2 \le \theta \sqrt{D}$,
    \end{enumerate}
    there is a $(d_1,d_2)$-biregular graph $H$ with $D_1$ vertices on the left and $D_2$ vertices on the right such that:
    \begin{align*}
        P_H(t) \ge (1-o_D(1))\cdot d_1 \cdot \exp(-\theta t/\sqrt{D})
    \end{align*}
    for $1 \le t \le \tau\sqrt{D}$ where $o_D(1)$ hides constant factors depending only on $\beta, \theta$ and $\tau$.}
\end{lemma}

We defer the proof of \Cref{lem:gadget} to \Cref{sec:gadget}, and prove our main theorem below.

\begin{theorem}[Formal version of \Cref{thm:main}] \label{thm:main-thm}
    For every $\beta\in(0,1/2]$, there are $d_0 = d_0(\beta) > 0$ and $\delta = \delta(\beta) > 0$ such that for all even $d_1,d_2\ge d_0$ with $1\ge\frac{d_1}{d_2} \ge \frac{\beta}{1-\beta}$, there is an infinite family of $(2d_1, 2d_2)$-biregular graphs $(Z_n)_{n\ge 1}$ with
    \[
        \abs*{\UN_{Z_n}(S)} \ge \delta \cdot d_1\cdot |S|,
        \quad
        \forall S \subseteq V(Z_n)\text{ with }|S| \le \frac{1}{d_1^3} \cdot \abs*{V\parens*{Z_n}} \mper
        \numberthis \label{eq:UNE-Z}
    \]
\end{theorem}
\begin{proof}
    The construction of $Z_n$ is based on taking the line product of $G_n$ from \Cref{fact:ramanujan-cayley} and a bipartite gadget graph $H$ from \Cref{lem:gadget} for suitably chosen parameters.

    Fix parameters $\tau = 18$, $\theta = \frac{40\tau}{\beta}$, and choose $d_0$ to be large enough such that $d_0 \geq \frac{8}{\beta} e^{\tau\theta}$ and such that for $D \geq 4d_0^2/\theta^2$, the $o_D(1)$ terms in \Cref{lem:gadget} and all subsequent occurrences in this proof are smaller than $0.1$.
    For $d_1,d_2 \ge d_0$, let $D_1 \coloneqq \ceil*{\frac{d_1+d_2}{\theta^2}}\cdot d_2$ and $D_2 \coloneqq \ceil*{\frac{d_1+d_2}{\theta^2}}\cdot d_1$, and define $D\coloneqq D_1 + D_2$.
    This choice of parameters satisfies the requirements of \Cref{lem:gadget}, and hence there is a $(d_1,d_2)$-biregular graph $H$ with $D_1$ left vertices, $D_2$ right vertices, and
    \[
        P_H(t) \ge 0.9 \cdot d_1 \cdot \exp(-\theta t/\sqrt{D})
        \numberthis \label{eq:PH-lower-bound}
    \]
    for $t \le \tau\sqrt{D}$.

    We choose $G_n$ as an $n$-vertex $D$-regular expander graph from \Cref{lem:ramanujan-deleted}.
    It remains to prove that $Z_n = G_n \Lining H$ has the claimed unique-neighbor expansion guarantee, and is indeed a $(2d_1,2d_2)$-biregular graph (which requires an appropriate ordering of $H$).

    Note that by \Cref{lem:ramanujan-deleted}, $G_n$ satisfies
    $2e(S) \leq 9\sqrt{D} \cdot |S|$ for all $S\subseteq V(G_n)$ with $|S| \leq \frac{1}{\sqrt{D}} n$, which satisfies the first condition in \Cref{lem:line-analysis} with $\eps = \frac{1}{\sqrt{D}}$ and $\gamma = 9\sqrt{D} = \frac{\tau}{2}\sqrt{D}$.
    Thus, it suffices to show that $P_H(t)$ satisfies the second condition in \Cref{lem:line-analysis} (a weaker lower bound):
    \begin{equation*}
        (\ref{eq:PH-lower-bound})
        \ge \frac{12\gamma}{t} = \frac{6\tau\sqrt{D}}{t},
        \quad \textnormal{ for } 1 \leq t \leq 2\gamma = \tau\sqrt{D} \mper
        \numberthis \label{eq:PH-condition}
    \end{equation*}
    This allows us to apply \Cref{lem:line-analysis} and the (stronger) lower  bound of $P_H(t)$ in \Cref{eq:PH-lower-bound} to get
    \begin{equation*}
        P_{Z_n}\Paren{\frac{1}{D^{3/2}}\cdot |V(Z_n)|} \geq \frac{1}{3} P_H(\tau\sqrt{D})
        \geq 0.3 e^{-\tau\theta} \cdot d_1
        \coloneqq \delta \cdot d_1
    \end{equation*}
    for constant $\delta = 0.3e^{-\tau\theta}$.
    Since $D \leq \frac{4}{\theta^2}(d_1+d_2)^2$, $\beta(d_1+d_2) \leq d_1$,  and $\beta\theta = 40\tau$, we have $\frac{1}{D^{3/2}} \geq \frac{1}{d_1^3}$.
    As $P_{Z_n}(t)$ is a decreasing function with $t$, this establishes the desired unique-neighbor expansion as articulated in \Cref{eq:UNE-Z}, finishing the proof of the theorem.

    Now, to establish \Cref{eq:PH-condition}, observe that the function $x e^{-x}$ is monotone increasing for $x \leq 1$ and monotone decreasing for $x \geq 1$, hence for $x\in[a,b]$, $x e^{-x} \geq \min\{a e^{-a}, b e^{-b}\}$.
    Thus, for $1 \leq t \leq \tau\sqrt{D}$,
    \begin{equation*}
        \frac{\theta t}{\sqrt{D}} \cdot e^{-\theta t/\sqrt{D}} \geq \min \Set{ \frac{\theta}{\sqrt{D}} \cdot e^{-\theta/\sqrt{D}},\ \tau\theta \cdot e^{-\tau\theta} }.
    \end{equation*}
    By using $d_1 \geq \beta(d_1+d_2)$, $d_1 + d_2 \geq \theta\sqrt{D} /2$ and the above, from \Cref{eq:PH-lower-bound} we get $P_H(t) \geq \frac{\sqrt{D}}{t} \cdot 0.45 \beta \theta \cdot \min\set{ e^{-\theta/\sqrt{D}}, \tau\sqrt{D} \cdot e^{-\tau\theta} }$.

    With our choice of parameters, $\beta \theta \geq 40\tau$ and $\sqrt{D} \geq 2d_0/\theta \geq \theta$ imply that $0.45\beta \theta \cdot e^{-\theta/\sqrt{D}} \geq 6\tau$.
    Furthermore, $\theta\sqrt{D} \geq d_1 + d_2 \geq 2d_0 \geq \frac{16}{\beta}e^{\tau\theta}$ implies that $0.45\beta \theta \cdot \tau\sqrt{D} e^{-\tau\theta} \geq 6\tau$.
    Therefore, we have established \Cref{eq:PH-condition}.

    Finally, we show that $Z_n$ is a $(2d_1, 2d_2)$-biregular graph.
    Since $G_n = \mathrm{Cay}(\Gamma,A)$ is a Cayley graph with generators $A$, each edge $\{u,v\}$ is labeled by group elements $a$ and $a^{-1}$ in $A$, i.e., $\{u,v\} = e_u^a = e_v^{a^{-1}}$.
    To construct the line product $G_n \Lining H$, we need a bijective map $\varphi$ between $A$ and $V(H) = L(H) \cup R(H)$ such that
    each pair $a, a^{-1} \in A$ gets assigned to the same side of $H$.
    This can be done as long as $d_1$ and $d_2$ are even.

    Let $L(Z_n) = \{e_v^a: v\in V(G_n),\ \varphi(a) \in L(H)\}$ and $R(Z_n) = \{e_v^a: v\in V(G_n),\ \varphi(a) \in R(H)\}$.
    First, $L(Z_n)$ and $R(Z_n)$ is a disjoint partition of $E(G_n) = V(Z_n)$ since $a, a^{-1}$ are assigned to the same side of $H$.
    Moreover, all edges of $Z_n$ are between $L(Z_n)$ and $R(Z_n)$, establishing bipartiteness of $Z_n$.
    Finally, observe that the degree of $e \in V(Z_n)$, an edge between $u$ and $v$, is $d_1$ in both $H(u)$ and $H(v)$ if $e\in L(Z_n)$, and $d_2$ in both if $e\in R(Z_n)$, which implies $(2d_1,2d_2)$-biregularity.
\end{proof}

%% file: spectral-radius.tex
\section{Non-backtracking matrix of subgraphs}  \label{sec:nb-specrad}

In this section, we bound the spectral radius of the nonbacktracking matrix of subgraphs of bipartite spectral expanders.
This gives us tight control over the degree profile of subgraphs, improving on bounds provided by classic tools like the expander mixing lemma~\cite{AC88}.


The following is a generalization of an analogous result of Kahale for regular graphs \cite[Theorem 1]{Kah95}.

\begin{theorem}[Formal version of \Cref{thm:subgraph-spectral-radius-main}] \label{thm:biregular-subgraph}
    Let $\eps \in (0,0.1)$, and let $3 \leq c \leq d$ be integers.
    Let $G = (L \cup R, E)$ be a $(c,d)$-biregular graph and $S \subseteq L \cup R$ such that $|S| \leq d^{-1/\eps} |L\cup R|$.
    Then, for any $t \geq 0$ such that
    \begin{equation*}
        \frac{1}{t} \geq \frac{1}{2} \Paren{\sqrt{\wt{\lambda}^2 - (\sqrt{c-1}+\sqrt{d-1})^2} + \sqrt{\wt{\lambda}^2 - (\sqrt{c-1}-\sqrt{d-1})^2} } \mcom
        \numberthis \label{eq:t-condition}
    \end{equation*}
    where $\wt{\lambda} = \max(\lambda_2(A_G), \sqrt{c-1}+\sqrt{d-1}) \cdot (1+O(\eps))$, we have
    \begin{equation*}
        H_{G[S]}(t) \succ 0 \mper
    \end{equation*}
\end{theorem}

To understand \Cref{eq:t-condition}, consider $\wt{\lambda} \approx \sqrt{c-1} + \sqrt{d-1}$.
Then, \Cref{eq:t-condition} simplifies to $\frac{1}{t} \gtrsim ((c-1)(d-1))^{1/4}$.
More specifically, suppose $\wt{\lambda} = (\sqrt{c-1} + \sqrt{d-1}) (1+\gamma/d)$, then denoting $\eta \coloneqq (\sqrt{c-1}+\sqrt{d-1})^2 (\frac{2\gamma}{d} + \frac{\gamma^2}{d^2})$, \Cref{eq:t-condition} becomes $\frac{1}{t} \geq \sqrt{\sqrt{(c-1)(d-1)} + \eta/4} + \frac{1}{2}\sqrt{\eta}$.
In particular, we have the following corollary.
    
\begin{corollary} \label{cor:t-condition-simplified}
    Let $3 \leq c \leq d$, and suppose $\wt{\lambda} = (\sqrt{c-1} + \sqrt{d-1}) (1+\gamma/d)$ for $\gamma \in [0, 2]$, then
    \begin{equation*}
        \frac{1}{t^2} \geq \sqrt{(c-1)(d-1)} \cdot (1 + 3\sqrt{\gamma})
    \end{equation*}
    implies \Cref{eq:t-condition}.
\end{corollary}

\parhead{Overview of the proof.}
We prove $H_{G[S]}(t) \succ 0$ by showing that $\iprod{f, H_{G[S]}(t) f} > 0$ for all $f: S\to \R$.
The way we prove $\iprod{f, H_{G[S]}(t) f} > 0$ is by relating it to a quadratic form against the matrix $H_G(t)$, which we can control via the spectrum of $G$.
In particular, we consider the depth-$\ell$ regular tree extension $T$ of $G[S]$, and for $f$ we define an appropriate function extension $f_t$ on the tree depending on $t$ (\Cref{def:function-extension}) such that $\iprod{f, H_{G[S]}(t) f} = \iprod{f_t, H_{T}(t) f_t}$.
The function $f_t$ additionally has the property that its $\ell_2$ mass on vertices $r$-far from $G[S]$ decays exponentially in $r$. 
At a high level, we use the tree extension as a proxy for the $\ell$-step neighborhood of $S$ in $G$, and this is made precise in \Cref{sec:folding-tree-extension} as we define a natural folded function $\wt{f}_t$ of $f_t$ into $G$ (\Cref{def:folded-function}).

This allows us to lower bound $\iprod{f, H_{G[S]}(t) f}$ by $\iprod{\wt{f}_t, H_{G}(t) \wt{f}_t}$ with some errors.
The errors can be bounded using the decay of $f_t$ from the definition, though this requires $t < ((c-1)(d-1))^{-1/4}$ (see \Cref{lem:function-extension-decay}).
Ignoring errors, it comes down to showing that
\begin{equation*}
    \frac{1}{t^4} - \frac{1}{t^2} \Paren{\lambda^2 - (c-1) - (d-1)} + (c-1)(d-1) > 0 \mcom
\end{equation*}
and we solve the quadratic formula in \Cref{lem:lambda-bound} and show that the above gives rise to \Cref{eq:t-condition}.
The full proof is presented in \Cref{sec:proof-of-spectral-radius}.

\subsection{Tree extensions}

We start with defining tree extensions of a graph.

\begin{definition}[Tree extension] \label{def:tree-extension}
    For a graph $G = (V, E)$, we say that $T = (V(T), E(T))$ is a tree extension of $G$ if $T$ is obtained by attaching a tree $T_r$ to each vertex $r \in V$, with $r$ being the root.
    Each vertex $x\in T$ belongs to a unique tree $T_r$ rooted at $r$.
    For any $x\in T$, we write $\depth(x)$ to be the distance between $x$ and the root of the tree containing $x$.

    Fix a tree extension $T$ of $G$, for functions $f, g: V(T) \to \R$, define $\iprod{f,g} = \sum_{x\in T} f(x) g(x)$ and $\normt{f}^2 = \sum_{x\in T} f(x)^2$.
\end{definition}

\begin{definition}[Function extension] \label{def:function-extension}
    Given a function $f : V(G) \to \R$, a tree extension $T$ of $G$, and parameter $t \in \R$, we define $f_{t}: V(T) \to \R$ to be the \emph{extension} of $f$ to $T$ such that for $x\in T$,
    \begin{equation*}
        f_{t}(x) = f(r) \cdot t^{\depth(x)},\quad \text{if } x\in T_r \mper
        \numberthis \label{eq:f-extension}
    \end{equation*}
\end{definition}

The following simple but crucial lemma establishes a relationship between $H_G$ and $H_{T}$, which also motivates the definition of $f_t$.

\begin{lemma} \label{lem:H-TG}
    Let $G$ be a graph and $T$ be any tree extension of $G$.
    Then, for any $t\in \R$ and $f: V(G) \to \R$,
    the extension $f_t: V(T) \to \R$ defined in \Cref{eq:f-extension} satisfies
    \begin{equation*}
        \Paren{H_{T}(t) \cdot f_t}(x) = 
        \begin{cases}
            (H_G(t) \cdot f)(x) & x\in V(G) \mcom \\
            0 & x \notin V(G) \mper
        \end{cases}
    \end{equation*}
\end{lemma}
\begin{proof}
    Recall that $H_G(t) = (D_G-\Id) t^2 - A_G t + \Id$.
    For $x \notin V(G)$, let $d(x)$ be its degree and let $r\in V(G)$ be the root of the tree containing $x$.
    Observe that $x$ has 1 parent (with value $f(r) t^{\depth(x)-1}$) and $d(x)-1$ children (with value $f(r) t^{\depth(x)+1}$) in the tree $T_r$.
    Thus,
    \begin{align*}
        \Paren{H_{T}(t) \cdot f_t}(x) &= ((d(x)-1)t^2 + 1) \cdot f(r) t^{\depth(x)} - t\cdot f(r) \Paren{t^{\depth(x)-1} + (d(x)-1) t^{\depth(x) + 1}} \\
        &= 0 \mper
    \end{align*}
    For $x\in V(G)$, let $d_G(x)$ be its degree in $G$ and $d_T(x)$ be its degree in $T$.
    Then, $x$ has $d_T(x) - d_G(x)$ children (with value $t \cdot f(x)$) in the tree $T_x$.
    \begin{align*}
        \Paren{H_{T}(t) \cdot f_t}(x)
        &= ((d_T(x)-1)t^2 + 1) \cdot f(x) - t \Bigparen{ (A_G f)(x) + (d_T(x)-d_G(x)) \cdot t f(x)} \\
        &= ((d_G(x)-1)t^2 + 1) \cdot f(x) - t (A_G f)(x) \\
        &= (H_G(t) \cdot f)(x) \mper
    \end{align*}
    This completes the proof.
\end{proof}


\subsection{Regular tree extensions of subgraphs}

For a subgraph $G[S]$ in a regular (or biregular) graph, we consider its regular tree extension.

\begin{definition}[Regular tree extension] \label{def:regular-tree-extension}
    Let $G = (V, E)$ be a $d$-regular graph, $S\subseteq V$, $\ell \in \N$, and consider the induced subgraph $G[S]$.
    We define the \emph{depth-$\ell$ regular tree extension} of $G[S]$ to be the tree extension $T$ of $G[S]$ where depth-$\ell$ trees are attached to vertices in $S$ such that the resulting graph is $d$-regular except for the leaves.
    Let $\Leaves(T)$ denote the set of leaves.

    Similarly, let $G = (L \cup R, E)$ be a $(c,d)$-biregular graph, $S \subseteq L \cup R$, and $\ell \in \N$.
    The depth-$\ell$ regular tree extension of $G[S]$ is the tree extension such that the resulting graph is $(c,d)$-biregular except for the leaves.
\end{definition}

We show that given a graph $G = (V,E)$ and $S \subseteq V$, for any function $f: S \to \R$ and its extension $f_t$ to the depth-$\ell$ regular tree extension of $G[S]$, the contribution from the leaves decays exponentially with $\ell$ when $t < ((c-1)(d-1))^{-1/4}$.

\begin{lemma}[Decay of $f_t$] \label{lem:function-extension-decay}
    Let $G = (L \cup R, E)$ be a $(c,d)$-biregular graph with $c\leq d$,
    let $S \subseteq L \cup R$, let $\ell \in \N$ be even, and let $T$ be the depth-$\ell$ regular tree extension of $G[S]$.
    Moreover, let $t\in \R$  such that $t^2 \sqrt{(c-1)(d-1)} = 1-\delta$ for some $\delta \in (0,1)$.
    Given any function $f: S \to \R$, let $f_t: V(T) \to \R$ be the function extension (as defined in \Cref{eq:f-extension}), and let $f_t^{=\ell}$ be $f_t$ restricted to the leaves of $T$. Then,
    \begin{equation*}
        \Normt{f_t^{=\ell}}^2 \leq \frac{2\delta}{e^{\delta \ell} - 1}\cdot \Normt{f_t}^2 \mper
    \end{equation*}
    Since the function $\frac{2x}{e^{x\ell} - 1}$ is monotone decreasing, we have for any $\delta' > \delta$,
    \begin{equation*}
        \Normt{f_t^{=\ell}}^2 \leq \frac{2\delta'}{e^{\delta' \ell} - 1}\cdot \Normt{f_t}^2 \mper
    \end{equation*}
\end{lemma}
\begin{proof}
    We will lower bound $\Normt{f_t}^2$ and upper bound the contribution from the leaves at depth $\ell$.
    Fix a vertex $r\in R$ (with $\deg_G(r) = d$) and consider the tree $T_r$ rooted at $r$.
    Let $\deg_{T_r}(r)$ be the degree of $r$ in $T_r$.
    The number of children of vertices in the tree alternates between $c-1$ and $d-1$ as we go down the tree.
    Thus, for an even integer $k \leq \ell$, the number of vertices in the $k$-th level is
    \begin{equation*}
        \deg_{T_r}(r) (c-1) \Paren{(c-1)(d-1)}^{\frac{k}{2}-1}
        = \frac{\deg_{T_r}(r)}{\deg_G(r)-1} \Paren{(c-1)(d-1)}^{\frac{k}{2}} \mper
        \numberthis \label{eq:even-layers}
    \end{equation*}
    The same argument shows that the above also holds for $r\in L$ (with $\deg_G(r) = c$).
    Thus, the contribution of the tree $T_r$ to $\Normt{f_t}^2$ can be lower bounded by the product of the following two terms:
    \begin{align*}
        &f(r)^2 \frac{\deg_{T_r}(r)}{\deg_G(r)-1} \\
         \sum_{\substack{0 \leq k \leq \ell \\ \text{$k$ even}}} t^{2k} ((c-1)(d-1))^{\frac{k}{2}}
        =& \sum_{i=0}^{\ell/2} (1-\delta)^{2i}
        =\frac{1 - (1-\delta)^{\ell+2}}{1 - (1-\delta)^2}
        \geq \frac{1 - (1-\delta)^\ell}{2\delta} \mper
    \end{align*}
    Next, the contribution from the leaves of $T_r$ to $\Normt{f_t^{=\ell}}^2$ is also given by \Cref{eq:even-layers}.
    Thus, we have
    \begin{equation*}
        \frac{\Normt{f_t^{=\ell}}^2}{\Normt{f_t}^2} \leq (1-\delta)^{\ell} \frac{2\delta}{1 - (1-\delta)^{\ell}} 
        \leq \frac{2\delta}{e^{\delta \ell} - 1} \mcom
    \end{equation*}
    using $(1-\delta)^\ell \leq e^{-\delta \ell}$, finishing the proof.
\end{proof}

\subsection{Folding regular tree extensions}
\label{sec:folding-tree-extension}

Given a regular tree extension $T$ of an induced subgraph $G[S]$, there is a natural folding into $G$ via breadth-first search from $S$.

\begin{definition}[Folding into $G$] \label{def:folding}
    Let $G = (V, E)$ be a $d$-regular or $(c,d)$-biregular graph, let $S \subseteq V$, and let $T$ be the depth-$\ell$ regular tree extension of $G[S]$.
    There is a natural homomorphism $\sigma: T \to G$ such that
    \begin{itemize}
        \item $\sigma(x) = x$ for all $x \in S$;
        \item $\deg_T(x) = \deg_G(\sigma(x))$ for all $x\in V(T) \setminus \Leaves(T)$;
        \item Two edges $\{x,y\}$ and $\{y,z\}$ in $T$ sharing a vertex are not mapped to the same edge in $E$, i.e., all edges in $T$ that map to the same edge in $E$ are vertex-disjoint.
    \end{itemize}
\end{definition}

\begin{definition}[Folded function] \label{def:folded-function}
    Fix a map $\sigma : T \to G$.
    Given any $f : V(T) \to \R$, we associate each vertex $v\in G$ with a function $f^v : V(T) \to \R$ such that for $x\in T$,
    \begin{align*}
        f^v(x) = \begin{cases}
            f(x) & \text{if $\sigma(x) = v$} \mcom \\
            0 & \text{otherwise} \mper
        \end{cases}
    \end{align*}
    We define the \emph{folded} function $\wt{f}: V(G) \to \R$ to be
    \begin{equation*}
        \wt{f}(v) = \Normt{f^v} \mper
    \end{equation*}
\end{definition}

\begin{observation} \label{obs:folded-function-norm}
    The $f^v$'s have disjoint support, thus $\normt{\wt{f}}^2 = \sum_{v\in G} \normt{f^v}^2 = \normt{f}^2$.
    More generally, let $\Gamma, \wt{\Gamma}$ be diagonal operators such that $(\Gamma f)(x) = \gamma(\deg_G(\sigma(x))) f(x)$ for $x\in T$ and $(\wt{\Gamma}g)(v) = \gamma(\deg_G(v)) g(v)$ for $v\in G$.
    Then, $\iprod{\wt{f}, \wt{\Gamma} \wt{f}} = \iprod{f, \Gamma f}$.
\end{observation}

We next prove the following useful lemma that relates the quadratic forms of $f$ and $\wt{f}$ with $A_G$.

\begin{lemma} \label{lem:folded-quadratic-form}
    Let $G = (V,E)$ be a $d$-regular or $(c,d)$-biregular graph, let $S\subseteq V$, and let $T$ be a regular tree extension of $G[S]$.
    For any $f : V(T) \to \R$ and its folded function $\wt{f} : V(G) \to \R$, we have
    \begin{equation*}
        \Iprod{f, A_{T} f} \leq \Iprod{\wt{f}, A_G \wt{f}} \mper
    \end{equation*}
\end{lemma}
\begin{proof}
    Recall from \Cref{def:folding} that the map $\sigma: T \to G$ satisfies that if $\{x,y\}$ is an edge in $T$, then $\{\sigma(x), \sigma(y)\} \in E$.
    Then,
    \begin{equation*}
        \Iprod{f, A_{T} f} = 2\sum_{\{x, y\} \in E(T)} f(x) f(y)
        = 2\sum_{\{u, v\} \in E(G)} \sum_{\{x, y\} \in E(T)}  \bone(\sigma(\{x,y\}) = \{u,v\}) \cdot f^{\sigma(x)}(x) f^{\sigma(y)}(y) \mper
    \end{equation*}
    Moreover, all edges in $T$ that map to the same edge are vertex-disjoint.
    Thus, for any $\{u, v\} \in E$, $\sum_{\{x, y\} \in E(T)}  \bone(\sigma(\{x,y\}) = \{u,v\}) \cdot f^{\sigma(x)}(x) f^{\sigma(y)}(y)$ can be expressed as an inner product between some permutations of $f^u$ and $f^v$, which is upper bounded by $\normt{f^u} \cdot \normt{f^v}= \wt{f}(u) \wt{f}(v)$ by Cauchy-Schwarz.
    Thus, we have
    \begin{equation*}
        \Iprod{f, A_{T} f}
        \leq 2\sum_{\{u, v\} \in E(G)} \wt{f}(u) \wt{f}(v)
        = \Iprod{\wt{f}, A_G \wt{f}} \mper
        \qedhere
    \end{equation*}
\end{proof}

\subsection{Proof of \texorpdfstring{\Cref{thm:biregular-subgraph}}{Theorem~\ref{thm:biregular-subgraph}}}
\label{sec:proof-of-spectral-radius}

Before we prove \Cref{thm:biregular-subgraph}, we first prove the following lemma for convenience.

\begin{lemma} \label{lem:lambda-bound}
    Let $3 \leq c \leq d \in \N$ and $\eps \in (0,1)$.
    Let $\lambda \geq \sqrt{c-1} + \sqrt{d-1}$ and $\wt{\lambda} = \lambda(1+\eps)$.
    Then, for all $x$ such that
    \begin{equation*}
        x \geq \frac{1}{2} \Paren{\sqrt{\wt{\lambda}^2 - (\sqrt{c-1}+\sqrt{d-1})^2} + \sqrt{\wt{\lambda}^2 - (\sqrt{c-1}-\sqrt{d-1})^2} } \mcom
    \end{equation*}
    we have
    \begin{equation*}
        x^4 - x^2 (\lambda^2(1+\eps) - (c+d-2)) + (c-1)(d-1) > 0 \mper
    \end{equation*}
\end{lemma}
\begin{proof}
    Denote $a \coloneqq c-1$ and $b \coloneqq d-1$ for convenience.
    Then, to show that $x^4 - x^2(\lambda^2(1+\eps) - a-b) + ab \geq 0$, it suffices to verify that
    \begin{equation*}
        x^2 > \frac{1}{2} \Paren{\lambda^2(1+\eps) - a-b} + \frac{1}{2} \sqrt{\Paren{\lambda^2(1+\eps) - a-b}^2 - 4ab} \mper
    \end{equation*}
    Squaring both sides of $x \geq \frac{1}{2} \Bigparen{\sqrt{\wt{\lambda}^2 - (\sqrt{a}+\sqrt{b})^2} + \sqrt{\wt{\lambda}^2 - (\sqrt{a}-\sqrt{b})^2} }$, we get
    \begin{align*}
        x^2 &\geq \frac{1}{2} \Paren{\wt{\lambda}^2 - a - b} + \frac{1}{2} \sqrt{(\wt{\lambda}^2 - (\sqrt{a}+\sqrt{b})^2) (\wt{\lambda}^2 - (\sqrt{a}-\sqrt{b})^2)} \\
        &= \frac{1}{2} \Paren{\wt{\lambda}^2 - a - b} + \frac{1}{2} \sqrt{\wt{\lambda}^4 - 2(a+b) \wt{\lambda}^2 + (a-b)^2} \\
        &= \frac{1}{2} \Paren{\wt{\lambda}^2 - a - b} + \frac{1}{2} \sqrt{(\wt{\lambda}^2 - a - b)^2 - 4ab} \mcom
    \end{align*}
    which completes the proof with $\wt{\lambda} = \lambda(1+\eps)$.
\end{proof}

\begin{proof}[Proof of \Cref{thm:biregular-subgraph}]
    We first verify that the assumption on $t$ (\Cref{eq:t-condition}) implies that
    \begin{equation*}
        t^2 \leq \frac{1-\eps}{\sqrt{(c-1)(d-1)}} \mper
        \numberthis \label{eq:t-ub}
    \end{equation*}
    Indeed, as $\wt{\lambda} = \lambda(1+O(\eps)) \geq (\sqrt{c-1} + \sqrt{d-1})(1+\eps)$, \Cref{eq:t-condition} implies that
    \begin{align*}
        \frac{1}{t^2} &\geq \frac{1}{4} \Paren{(\sqrt{c-1}+\sqrt{d-1})^2(1+\eps)^2 - (\sqrt{c-1}-\sqrt{d-1})^2} \\
        &\geq \sqrt{(c-1)(d-1)} + \frac{1}{2}  (\sqrt{c-1}+\sqrt{d-1})^2 \eps
    \end{align*}
    which implies \Cref{eq:t-ub}.

    We would like to show that $\iprod{f, H_{G[S]}(t) f} > 0$ for any function $f: S \to \R$.
    Let $\ell = \ceil{\frac{1}{2\eps}}$ be an even integer and let $T$ be the depth-$\ell$ regular tree extension of $G[S]$ (\Cref{def:regular-tree-extension}).
    Let $f_t: V(T) \to \R$ be the function extension of $f$ to $T$ with parameter $t$.
    By \Cref{lem:H-TG}, we have
    \begin{align*}
        \Iprod{f, H_{G[S]}(t) f} = \Iprod{f_t, H_T(t) f_t}
        = \Iprod{f_t, ((D_T-\Id) t^2 - tA_T + \Id) f_t} \mper
    \end{align*}
    Note that all internal vertices $x \in T \setminus \Leaves(T)$ have degree $c$ or $d$ while the leaves have degree $1$.
    Let $D_{T}'$ be the diagonal matrix such that the leaves have the ``correct'' degree, i.e., for $x \in \Leaves(T)$ in the tree $T_r$ rooted at $r \in S$, $D_{T}'[x,x] = \deg_G(r)$ (since $\ell$ is even).
    Then, by \Cref{eq:t-ub}, \Cref{lem:function-extension-decay} states that $f_t^{=\ell}$ decays with a factor $\frac{2\eps}{e^{\eps \ell}-1} \leq 4\eps$, thus
    \begin{equation*}
        \Iprod{f_t, (D_T-\Id) f_t} 
        = \Iprod{f_t, (D_{T}'- \Id) f_t} - \Iprod{f_t^{=\ell}, (D_{T}'-\Id) f_t^{=\ell}}
        \geq \Iprod{f_t, (D_{T}'- \Id) f_t} \Paren{1 - 4\eps} \mper
    \end{equation*}

    Consider the folded function $\wt{f}_t: V(G) \to \R$ as defined in \Cref{def:folded-function}.
    By \Cref{obs:folded-function-norm}, we have $\Iprod{f_t, D_{T}' f_t} = \iprod{\wt{f}_t, D_{G} \wt{f}_t}$ and $\normt{f_t}^2 = \normt{\wt{f}_t}^2$.
    Moreover, by \Cref{lem:folded-quadratic-form}, $\iprod{f_t, A_{T} f_t} \leq \Iprod{\wt{f}_t, A_G \wt{f}_t}$.
    Thus,
    \begin{align*}
        \Iprod{f, H_{G[S]}(t) f}
        &\geq t^2 \Iprod{\wt{f}_t, (D_G-\Id) \wt{f}_t} (1-4\eps) - t \Iprod{\wt{f}_t, A_G \wt{f}_t} + \Normt{\wt{f}_t}^2 \\
        &\geq (1-4\eps)\Iprod{\wt{f}_t, \Paren{t^2(D_G - \Id) + \Id} \wt{f}_t} - t\Iprod{\wt{f}_t, A_G \wt{f}_t} \mper
        \numberthis \label{eq:f-Hf-lower-bound}
    \end{align*}
    We would like to show that the above is non-negative.
    Denote $\Gamma_G \coloneqq t^2(D_G - \Id) + \Id$,
    and $\gamma_1 \coloneqq t^2(c-1) + 1$ and $\gamma_2 \coloneqq t^2(d-1) + 1$.
    Note that $\gamma_2 \geq \gamma_1 > 0$ as we assume that $c \leq d$.
    Since $G$ is a $(c,d)$-biregular graph, $\Gamma_G$ and $A_G$ have the following block structure,
    \begin{equation*}
        \Gamma_G = \begin{pmatrix}
            \gamma_1 \Id & 0 \\ 0 & \gamma_2 \Id
        \end{pmatrix} \mcom
        \quad
        A_G = \begin{pmatrix}
            0 & A_{L,R} \\ A_{L,R}^\top & 0
        \end{pmatrix} \mper
    \end{equation*}
    In particular,
    \begin{equation*}
        (1-4\eps)\Gamma_G - tA_G = \Gamma_G^{1/2} \Paren{(1-4\eps)\Id - t \cdot \Gamma_G^{-1/2}A_G \Gamma_G^{-1/2}} \Gamma_G^{1/2}
        = \Gamma_G^{1/2} \Paren{(1-4\eps)\Id - \frac{t}{\sqrt{\gamma_1\gamma_2}} A_G} \Gamma_G^{1/2} \mper
    \end{equation*}
    Then, denoting $g \coloneqq \Gamma_G^{1/2} \wt{f}_t$, we can write \Cref{eq:f-Hf-lower-bound} as
    \begin{equation*}
        \Iprod{f, H_{G[S]}(t)f} \geq \Iprod{\wt{f}_t, \Paren{(1-4\eps)\Gamma_G - tA_G} \wt{f}_t}
        = (1-4\eps)\Normt{g}^2 - \frac{t}{\sqrt{\gamma_1\gamma_2}}\Iprod{g, A_G g} \mper
        \numberthis \label{eq:f-Hf-lb2}
    \end{equation*}

    Next, we upper bound $\iprod{g, A_G g}$.
    For any $(c,d)$-biregular graph, the (normalized) top eigenvector of $A_G$ is $\frac{1}{\sqrt{2|E|}} D_G^{1/2} \vec{1}$ with eigenvalue $\sqrt{cd}$.
    Thus,
    \begin{equation*}
        \Iprod{g, A_G g} \leq \frac{\sqrt{cd}}{2|E|} \Iprod{g, D_G^{1/2}\vec{1}}^2 + \lambda \Normt{g}^2 \mcom
    \end{equation*}
    where $\lambda = \max(\lambda_2(A_G), \sqrt{c-1}+\sqrt{d-1})$ is the second eigenvalue.

    Since $T$ has depth $\ell$, the support of $\wt{f}_t$ (and $g$) must be contained in $B \coloneqq \{v\in V(G): \dist(v, S) \leq \ell \}$.
    We have $|B| \leq |S| \sum_{i=0}^\ell d^i \leq |S| d^{\ell+1}$.
    Thus, by Cauchy-Schwarz,
    \begin{equation*}
        \frac{\sqrt{cd}}{2|E|} \Iprod{g, D_G^{1/2}\vec{1}}^2
        \leq \frac{\sqrt{cd}}{2|E|} \cdot d|B| \cdot \Normt{g}^2
        \leq d^{-1/4\eps} \Normt{g}^2
        \le \eps \Normt{g}^2 \mcom
    \end{equation*}
    since $|S| \leq d^{-1/\eps}|L \cup R|$, $\ell = \ceil{\frac{1}{2\eps}}$, $|E| = c|L| = d|R|$, and $\eps \leq 0.1$ (note that $d^{-1/4\eps} \leq \eps$ for all $d \geq 3$ and $\eps \leq 0.1$).

    Thus, $\iprod{g, A_G g} \leq (\lambda + \eps) \Normt{g}^2 \leq \lambda (1+\eps) \Normt{g}^2$, and from \Cref{eq:f-Hf-lb2},
    \begin{equation*}
        \Iprod{f, H_{G[S]}(t)f} \geq \frac{1}{\sqrt{\gamma_1\gamma_2}}\Paren{(1-4\eps) \sqrt{\gamma_1\gamma_2} - t\lambda(1 + \eps)} \mper
    \end{equation*}
    As $\frac{1+\eps}{1-4\eps} \leq 1 +5\eps$, to prove that the above is positive, it suffices to prove that $t^2 \lambda^2 (1+5\eps) < \gamma_1\gamma_2 = (t^2(c-1)+1)(t^2(d-1)+1)$, or equivalently,
    \begin{equation*}
        \frac{1}{t^4} - \frac{1}{t^2} \Paren{\lambda^2(1+5\eps) - (c-1) - (d-1)} + (c-1)(d-1) > 0 \mper
    \end{equation*}
    With $\wt{\lambda} = \lambda(1+5\eps)$ and the assumption on $t$ (\Cref{eq:t-condition}), the above holds via \Cref{lem:lambda-bound}.
\end{proof}

%% file: subgraph-density.tex
\section{Expansion and density of subgraphs}    \label{sec:subgraph-density}

The spectral radius of the nonbacktracking matrix of a subgraph of a bipartite graph imposes constraints on the left and right degrees, articulated by the following.



\begin{theorem}[Subgraph density in (near-)Ramanujan graphs; restatement of \Cref{thm:subgraph-density-main}] \label{thm:expansion}
    Let $3 \leq c \leq d$ be integers, $\gamma \in [0,1]$, and $\eps \in (0,0.1)$.
    Let $G = (L \cup R, E)$ be a $(c,d)$-biregular graph such that $\lambda_2(G) \leq (\sqrt{c-1} + \sqrt{d-1})(1 + \gamma/d)$.
    Then, there exists $\delta = \delta(\eps,c,d) > 0$ such that for every $S_1 \subseteq L$ and $S_2 \subseteq R$ with $|S_1| + |S_2| \leq \delta |L \cup R|$,
    the left and right average degrees $d_1 = \frac{|E(S_1,S_2)|}{|S_1|}$ and $d_2 = \frac{|E(S_1,S_2)|}{|S_2|}$ in the induced subgraph $G[S_1 \cup S_2]$ must satisfy
    \begin{equation*}
        (d_1-1) (d_2-1) \leq \sqrt{(c-1)(d-1)} \cdot (1+O(\eps+\sqrt{\gamma})) \mper
    \end{equation*}
\end{theorem}

\Cref{thm:expansion} is a direct consequence of \Cref{thm:biregular-subgraph} and the following lemma:
\begin{lemma} \label{lem:bipartite-avg-degree}
    Let $G = (L \cup R, E)$ be a bipartite graph, and let the left and right average degrees be $d_1 = \frac{|E|}{|L|}$ and $d_2 = \frac{|E|}{|R|}$, respectively.
    Then, for any $t\in (-1,1) \setminus \{0\}$ such that $H_G(t) \succeq 0$, we have
    \begin{equation*}
        (d_1-1) (d_2-1) \leq \frac{1}{t^2} \mper
    \end{equation*}
\end{lemma}

\begin{proof}
    We can assume that $d_1, d_2 > 1$, otherwise the statement holds trivially with $|t| < 1$.
    Let $x$ be the vector such that for $u\in L \cup R$,
    \begin{equation*}
        x_u = \begin{cases}
            1 & u\in L, \\
            \alpha & u\in R,
        \end{cases}
    \end{equation*}
    where $\alpha \in \R$ will be determined later.
    Recall that $H_{G}(t) = (D_{G} - \Id)t^2 - t A_{G} + \Id$.
    Since $|E| = d_1|L| = d_2|R|$, we have $x^\top D_{G} x = d_1|L| + \alpha^2 d_2|R| = (1+\alpha^2) d_1 |L|$ and $x^\top A_{G} x = 2 \alpha |E| = 2\alpha d_1 |L|$, and substituting $|R| = \frac{d_1}{d_2}|L|$ we get
    \begin{align*}
        x^\top H_{G}(t) x &= x^\top \Paren{(D_{G} - \Id)t^2 - t A_{G} + \Id} x \\
        &= t^2 \Paren{(d_1-1)|L| + \alpha^2 (d_2-1) |R|} - t \cdot 2 \alpha  d_1|L| + \Paren{|L| + \alpha^2 |R|} \\
        &= |L| \Paren{(d_1-1)t^2 - 2t\alpha d_1 + 1} + |R| \Paren{\alpha^2 (d_2-1)t^2 + \alpha^2} \\
        &= |L| \Paren{ (d_1-1)t^2 + \alpha^2 d_1 t^2 - 2t \alpha d_1 + 1 + \frac{d_1}{d_2} \cdot \alpha^2(1-t^2) } \mper
    \end{align*}
    Then, $H_{G}(t) \succeq 0$ and $t \in (-1,1)$ imply that
    \begin{equation*}
        \frac{1}{d_2} \geq \frac{1}{1-t^2} \Paren{- \frac{(d_1-1)t^2 + 1}{d_1 \alpha^2} + \frac{2t}{\alpha} - t^2} \mper
    \end{equation*}
    To maximize the right-hand side, we choose $\frac{1}{\alpha} = \frac{d_1 t}{(d_1-1)t^2+1}$, which gives
    \begin{gather*}
        \frac{1}{d_2} \geq \frac{1}{1-t^2} \Paren{\frac{d_1 t^2}{(d_1-1)t^2+1} - t^2}
        = \frac{1}{1-t^2} \cdot \frac{t^2(d_1-1)(1-t^2)}{(d_1-1)t^2+1}
        = \frac{1}{1 + \frac{1}{(d_1-1)t^2}} \\
        \implies d_2 \leq 1 + \frac{1}{(d_1-1)t^2} \mper
    \end{gather*}
    Rearranging the above gives $(d_1-1)(d_2-1) \leq 1/t^2$.
\end{proof}

We can now prove \Cref{thm:expansion}.

\begin{proof}[Proof of \Cref{thm:expansion}]
    We consider the induced subgraph $G[S_1 \cup S_2]$.
    By \Cref{thm:biregular-subgraph}, we can choose $t$ such that $\frac{1}{t^2} \approx \sqrt{(c-1)(d-1)}$ and $H_{G[S_1\cup S_2]}(t) \succeq 0$ ---
    specifically, set $\delta = d^{-d/\eps^2}$ (depending only on $\eps,d$) and $\wt{\lambda} = (\sqrt{c-1}+\sqrt{d-1})(1 + \gamma/d) \cdot (1+O(\eps^2/d))$,
    then by \Cref{cor:t-condition-simplified}, $H_{G[S_1\cup S_2]} \succeq 0$ as long as
    \begin{equation*}
        \frac{1}{t^2} \geq \sqrt{(c-1)(d-1)} \cdot \Paren{1 + O\Paren{\sqrt{\gamma+\eps^2}}}
        = \sqrt{(c-1)(d-1)} \Paren{1+ O(\eps + \sqrt{\gamma})} \mper
    \end{equation*}
    Plugging the above into \Cref{lem:bipartite-avg-degree} completes the proof.
\end{proof}

As a corollary of \Cref{thm:expansion}, we recover the following result of Asherov and Dinur~\cite{AD23} proved for Ramanujan graphs, which further extends to near-Ramanujan graphs.

\begin{corollary} \label{cor:ve-near-ramanujan}
    Let $3 \leq c \leq d$ be integers, $\gamma \in [0,1]$, and $\eps \in (0,0.1)$.
    Let $G = (L \cup R, E)$ be a $(c,d)$-biregular graph such that $\lambda_2(G) \leq (\sqrt{c-1} + \sqrt{d-1})(1 + \gamma/d)$.
    Then, there exists $\delta = \delta(\eps,c,d) > 0$ such that for every $S \subseteq L$ of size $|S| \leq \delta |L|$,
    \begin{equation*}
        \frac{c|S|}{|N(S)|} \leq 1 + (1+O(\eps + \sqrt{\gamma})) \sqrt{\frac{d-1}{c-1}} \mper
    \end{equation*}
\end{corollary}

%% file: un-lossless.tex
\section{Unique-neighbor expanders with lossless small-set expansion}   \label{sec:ss-une}

In this section, we prove \Cref{thm:main-2}.
\restatetheorem{thm:main-2}

The graphs from \Cref{thm:main-2} are constructed by taking the tripartite line product of a tripartite graph on vertex set $L\cup M\cup R$ where $L\cup M$ and $M\cup R$ are (different) near-Ramanujan biregular graphs with no short bicycles, with suitably chosen degrees.
In particular, from the works of \cite{MOP20,OW20}, we have the following.
\begin{theorem}[Special case of main theorem in \cite{OW20}]
    For every $c,d\ge 3$ and $\gamma > 0$, there is an explicit construction of an infinite family of $(c,d)$-biregular graphs $(G_n)_{n\ge 1}$ where $\lambda_2(G_n) \le (\sqrt{c-1} + \sqrt{d-1})\cdot(1 + \gamma)$, and $G_n$ has no bicycle on $o(\sqrt{\log|V(G_n)|})$ vertices.
\end{theorem}

\Cref{thm:main-2} is then a consequence of choosing the graphs between $L$ \& $M$, and $M$ \& $R$ using the above theorem in conjunction with the forthcoming \Cref{thm:2-sided-un-sse}.

\begin{remark}
    Replacing the above choice of base graph with one equipped with a group action and improved bicycle-free radius would result in a group action for the construction in \Cref{thm:main-2} as well as improved parameters.
    It is plausible that the construction of \cite{BFGRKMW15} satisfies these properties, but we do not prove it in this work.
\end{remark}

\subsection{Lossless expansion in high-girth graphs}

We first define the notion of \emph{bicycles} from \cite{MOP20}.

\begin{definition}[Excess] \label{def:excess}
    Given a graph $G = (V, E)$, its \emph{excess} is $\exc(G) = |E| - |V|$.
    In particular, a graph with excess $0$ and $1$ is called cyclic and bicyclic respectively.
\end{definition}

We will also need the following refinement of the well-known irregular Moore bound for graphs.
\begin{theorem}[Generalized Moore bound; formal version of \Cref{thm:moore-bound-main}] \torestate{\label{thm:moore-bound}
    Suppose $G$ is a graph on $n$ vertices, and let $\rho = \lambda_1(B_G)$ be the spectral radius of its non-backtracking matrix $B_G$.
    Suppose $\rho > 1$, then $G$ contains a cycle of size at most $2(\floor{\log_{\rho} n} + 1)$ and a bicycle of size at most $3(\floor{\log_{\rho} 2n} + 1)$.}
\end{theorem}
We defer the proof of \Cref{thm:moore-bound} to \Cref{app:exact-moore-bound}.
As mentioned in \Cref{rem:generalized-moore-bound},
for a graph $G$ with average degree $d$, $\rho(B_G)$ is at least $d-1$.
This follows from the fact that $\vec{1}^\top H_G(\frac{1}{d-1}) \vec{1} = 0$ and \Cref{lem:spectral-radius}.
Therefore, \Cref{thm:moore-bound} (the cycle case) is potentially stronger than the girth guarantee of $2\log_{d-1} n$ from the classical Moore bound.

Finally, we will need the following statement about the expansion of small sets in graphs with no short cycles or bicycles, which generalizes \cite[Theorem 10]{Kah95}.
\begin{lemma}[Expansion of small sets] \label{lem:small-set-expansion}
    Let $G = (L \cup R, E)$ be a $d$-left-regular bipartite graph, and let $\eps \in (0,1)$ such that $\eps(d-1) > 1$.
    Suppose $G$ has no cycle of length at most $g$, then for all $S \subseteq L$ with $|S| \leq \frac{1}{d+1} (\eps(d-1))^{\frac{1}{4}g-\frac{1}{2}}$ we have $|N_G(S)| \geq (1-\eps) d|S|$.

    Similarly, suppose $G$ has no bicycle of length at most $g$, then for all $S \subseteq L$ with $|S| \leq \frac{1}{2(d+1)} (\eps(d-1))^{\frac{1}{6}g-\frac{1}{2}}$ we have $|N_G(S)| \geq (1-\eps) d|S|$.
\end{lemma}
\begin{proof}
    Let $T \coloneqq N_G(S) \subseteq R$.
    Suppose $S$ does not expand losslessly, i.e., $|T| < (1-\eps) d |S|$.
    Then, the subgraph $G[S \cup T]$ must have right average degree at least $\frac{d|S|}{(1-\eps)d|S|} \geq \frac{1}{1-\eps} \geq 1+\eps$.
    Let $\rho > 0$ be the spectral radius of the non-backtracking matrix $B_{G[S \cup T]}$ so that $H_{G[S \cup T]}(1/\rho) \succeq 0$.
    Then, applying \Cref{lem:bipartite-avg-degree}, we have
    \begin{equation*}
        1+\eps \leq 1 + \frac{\rho^2}{(d-1)}
        \implies \rho \geq \sqrt{\eps(d-1)} \mper
    \end{equation*}
    Next, by \Cref{thm:moore-bound}, $G[S\cup T]$ must contain a cycle of size at most
    \begin{equation*}
        2\log_{\rho} (|S|+|T|) + 2
        \leq \frac{2\log((d+1)|S|)}{\log \sqrt{\eps(d-1)}} + 2 \mper
    \end{equation*}
    Suppose $|S| \leq \frac{1}{d+1} (\eps(d-1))^{\frac{1}{4}(g-2)}$, then there exists a cycle of length at most $g$, which is a contradiction.

    Similarly, by \Cref{thm:moore-bound}, $G[S\cup T]$ must contain a bicycle of size at most
    \begin{equation*}
        3 \log_{\rho}(2(|S|+|T|)) + 3 
        \leq \frac{3\log(2(d+1)|S|)}{\log \sqrt{\eps(d-1)}} + 3 \mper
    \end{equation*}
    Suppose $|S| \leq \frac{1}{2(d+1)} (\eps(d-1))^{\frac{1}{6}(g-3)}$, then there exists a bicycle of length at most $g$, which is a contradiction.
\end{proof}

\subsection{Tripartite line product}
\label{sec:tripartite-line-product}

We now define a generalization of the line product.


\begin{definition}[Tripartite line product] \label{def:tripartite-line}
    Let $G = (L \cup M \cup R, E_1 \cup E_2)$ be a tripartite graph consisting of a $(K_1, D_1)$-biregular graph $G^{(1)} = (L \cup M, E_1)$ and a $(D_2,K_2)$-biregular graph $G^{(2)} = (M \cup R, E_2)$.
    Let $H$ be a bipartite graph on vertex set $[D_1] \cup [D_2]$.
    The tripartite line product $G \Lining H$ is the bipartite graph on vertex set $L \cup R$ and edges obtained by placing a copy of $H$ on the neighbors of $v$ for each $v \in M$.
\end{definition}

\begin{remark} \label{rem:generalize-line-product}
    Note that \Cref{def:tripartite-line} is indeed a generalization of the line product in \Cref{def:line-product} (where $H$ is bipartite).
    To see this, consider a $D$-regular graph $G$ and define a tripartite graph $G'$ as follows: set $M = V(G)$, set $L, R$ to be a partition of $E(G)$, and for $v\in M$ and $e = L \cup R$, connect $\{v, e\}$ if and only if $v \in e$.
    Note that in this case $K_1 = K_2 = 2$.
    If $L, R$ satisfy that each $v\in M$ has $D_1$ neighbors in $L$ and $D_2$ neighbors in $R$ (with $D = D_1 + D_2$), then $G' \Lining H$ is exactly the same as $G \Lining H$ in \Cref{def:line-product}.
\end{remark}

We now prove \Cref{thm:main-2}; we use the tripartite line product to construct two-sided unique-neighbor expanders where we can additionally guarantee that small enough sets expand losslessly.

\begin{theorem}[Formal version of \Cref{thm:main-2}]  \label{thm:2-sided-un-sse}
    Suppose for some $c,d \geq 3$ and all $\gamma > 0$, there exists an explicit infinite family of $(c,d)$-biregular near-Ramanujan graphs $(F_n)_{n\geq 1}$ with $\lambda_2(F_n) \leq (\sqrt{c-1} + \sqrt{d-1}) (1 + \gamma)$ and has no bicycle of size at most $g_n = \omega_n(1)$.

    For every $\beta\in(0,1/2]$ and $\eps \in (0,1)$, there are $K = K(\eps)$, $d_0 = d_0(\beta,\eps)$ and $\delta = \delta(\beta,\eps) > 0$ such that for all $d_1,d_2\ge d_0$ which are multiples of $K$ and satisfy $1\ge\frac{d_1}{d_2} \ge \frac{\beta}{1-\beta}$, the following holds:
    there exists $\mu = \mu(\beta,\eps,d_1,d_2) > 0$ such that there is an infinite family of $(d_1, d_2)$-biregular graphs $(Z_n)_{n\ge 1}$ where
    \begin{equation*}
        \abs*{\UN_{Z_n}(S)} \ge \delta \cdot d_1\cdot |S| \mcom
        \quad
        \forall S \subseteq V(Z_n) \text{ with }|S| \le \mu \cdot \abs*{V(Z_n)} \mper
    \end{equation*}
    Moreover,
    \begin{equation*}
    \begin{aligned}
        &\abs*{N_{Z_n}(S)} \ge (1-\eps) \cdot d_1\cdot |S| \mcom
        \quad \forall S \subseteq L(Z_n) \text{ with }|S| \le \exp(g_n) \mcom \\
        &\abs*{N_{Z_n}(S)} \ge (1-\eps) \cdot d_2\cdot |S| \mcom
        \quad \forall S \subseteq R(Z_n) \text{ with }|S| \le \exp(g_n) \mper
    \end{aligned}
    \end{equation*}
\end{theorem}

\begin{proof}
    The construction of $Z_n$ is based on taking the tripartite line product of some $G_n$ and a bipartite gadget graph $H$ from \Cref{lem:gadget} with suitably chosen parameters.

    \begin{itemize}
        \item Let $K = K_1 = K_2 = \frac{1000}{\eps^2}$.
        
        \item Let $d_0 \geq C_0 K$ where $C_0 = C_0(\eps,\beta)$ is some large enough constant chosen later.
        

        \item Let $\wt{d}_1 \coloneqq d_1/ K$ and $\wt{d}_2 \coloneqq d_2 / K$, both at least $C_0$.

        \item Let $\theta \coloneqq \frac{C}{\eps} \sqrt{K /\beta}$ (depending only on $\eps,\beta$) where $C$ is a universal constant.

        \item Let $D_1 \coloneqq \ceil*{\frac{\wt{d}_1+\wt{d}_2}{\theta^2}}\cdot \wt{d}_2$ and $D_2 \coloneqq \ceil*{\frac{\wt{d}_1+\wt{d}_2}{\theta^2}}\cdot \wt{d}_1$, and define $D \coloneqq D_1 + D_2$.
    \end{itemize}

    Note that $1 \geq \wt{d}_1 / \wt{d}_2 = d_1 / d_2 \geq \frac{\beta}{1-\beta}$.
    One can verify that $K D_1 \leq 2\eps^2 \wt{d}_2^2/C^2$ and $K D_2 \leq 2\eps^2 \wt{d}_1^2/C^2$.

    The above choice of parameters satisfy the requirements of \Cref{lem:gadget}.
    Thus, applying \Cref{lem:gadget} with parameters $\theta$ and $\tau=1$, there is a $(\wt{d}_1,\wt{d}_2)$-biregular graph $H$ with $D_1$ left vertices and $D_2$ right vertices such that
    \begin{equation*}
        \min_{S\subseteq V(H): 1 \le |S| \le t} \frac{|\UN_H(S)|}{|S|}  \geq (1-o_D(1)) \cdot \wt{d}_1 \cdot \exp(-\theta t/\sqrt{D})
        \numberthis \label{eq:gadget-expansion}
    \end{equation*}
    for $1 \leq t \leq \sqrt{D}$.
    We can set $C_0$ large enough (depending only on $\beta,\theta$ which only depend on $\eps,\beta$) such that the $o_D(1)$ term is at most $0.1$.

    For the tripartite base graph $G_n = (L \cup M \cup R, E_1 \cup E_2)$, we construct $G_n^{(1)} = (L \cup M, E_1)$ and $G_n^{(2)} = (M \cup R, E_2)$ to be $(K_1, D_1)$ and $(D_2, K_2)$-biregular near-Ramanujan graphs respectively, i.e., $\lambda_2(G_n^{(1)}) \leq (\sqrt{K_1-1} + \sqrt{D_1-1})(1+0.01/D_1)$ and $\lambda_2(G_n^{(2)}) \leq (\sqrt{K_2-1} + \sqrt{D_2-1}) (1+ 0.01/D_2)$, along with the small-bicycle-free assumption.


    \parhead{Unique-neighbor expansion.}
    Next, we analyze the vertex expansion of a subset $S \subseteq L(Z_n)$ in the product graph $Z_n$.
    Recall that $L(Z_n) = L$.
    Let $U \coloneqq N_{G_n}(S) \subseteq M$ be the neighbors of $S$ in $G_n^{(1)}$, and we partition $U$ into $U_\ell \coloneqq \{v\in U: |E_1(v, S)| \leq \sqrt{D} \}$ (the ``low $S$-degree'' vertices) and $U_h \coloneqq U \setminus U_\ell$ (the ``high $S$-degree'' vertices).
    Consider the bipartite subgraph induced by $S \cup U_h$.
    By definition, the average right-degree in $G_n^{(1)}[S \cup U_h]$ is at least $\sqrt{D}$.
    By the upper bound on $\lambda_2(G_n^{(1)})$, we can apply \Cref{thm:expansion} and bound the average left-degree by
    \begin{equation*}
        d_{\text{left}}(S,U_h) \leq 1 + \frac{\sqrt{(K_1-1)(D_1-1)}}{\sqrt{D}-1} \cdot 1.1
        \leq 1 + 1.2\sqrt{K} \mcom
    \end{equation*}
    as long as $|S| \leq \mu|L|$ for some $\mu = \mu(K, D_1) > 0$ (depending only on $\eps,\beta, d_1, d_2$).
    For any $K \geq 100$, the above is at most $0.2 K$.
    Thus, we know that $|E_1(S, U_\ell)| \geq 0.8 K|S|$, i.e., a constant fraction of edges incident to $S$ go to $U_{\ell}$.
    This also implies that $|U_h| \leq 0.2|U|$.

    For each $v\in U$, let $S_v \subseteq S$ be the vertices in $S$ incident to $v$.
    Consider the gadget $H$ placed on $v$, and let $T_v \subseteq R$ be the set of unique-neighbors of $S_v$ in the gadget.
    Further, let $\wt{T} \coloneqq \bigcup_{v\in U} T_v$.
    Note that each vertex in $\wt{T}$ is a unique-neighbor \emph{within some gadget}, but there may be edges coming from other gadgets, so not all of $\wt{T}$ are unique-neighbors of $S$ in the final product graph.
    Our goal is to show that a large fraction of $\wt{T}$ are unique-neighbors of $S$.
    
    We will analyze the induced subgraph $G_n^{(2)}[U \cup \wt{T}]$, and we claim that a large fraction of $\wt{T}$ are unique-neighbors of $U$ in $G_n^{(2)}$, thus are also unique-neighbors of $S$ in $Z_n$.
    We first lower bound the left average degree of $G_n^{(2)}[U \cup \wt{T}]$.
    For each $v\in U_{\ell}$, we have $1 \leq |S_v| \leq \sqrt{D}$, and by the expansion profile of the gadget (\Cref{eq:gadget-expansion}), $v$ has degree at least
    \begin{equation*}
        |S_v| \cdot 0.9 \cdot \wt{d}_1 \cdot \exp(-\theta |S_v|/\sqrt{D})
        \geq 0.9\cdot \wt{d}_1  \cdot \min\Set{e^{-\theta/\sqrt{D}}, \sqrt{D} e^{-\theta}}
    \end{equation*}
    in $G_n^{(2)}[U \cup \wt{T}]$.
    Since $\theta$ depends only on $\eps,\beta$, we choose $C_0 = C_0(\eps,\beta)$ to be large enough (thus also $D$) such that the above is at least $0.8 \cdot \wt{d}_1$.

    Next, for $v\in U_h$, we have no control over its degree in $G_n^{(2)}[U \cup \wt{T}]$.
    However, since $|U_h| \leq 0.2|U| \leq \frac{1}{4}|U_{\ell}|$, we have
    \begin{equation*}
        d_{\text{left}}(U, \wt{T}) \geq \frac{0.8\wt{d}_1 \cdot |U_\ell|}{|U_{\ell}|+ |U_h|} 
        \geq 0.64\cdot \wt{d}_1 \mper
    \end{equation*}
    Then, for $|S| \leq \mu|L|$ where $\mu$ is small enough, we have $|U| \leq \mu'|M|$ where $\mu'$ (depending on $\eps,\beta,K,D_2$) is small enough to apply \Cref{thm:expansion} and conclude that the right average degree
    \begin{equation*}
        d_{\text{right}}(U, \wt{T}) \leq 1 + \frac{\sqrt{(K_2-1)(D_2-1)}}{0.64 \cdot \wt{d}_1-1} \cdot O(1)
        \leq 1.1 \mcom
    \end{equation*}
    since $\wt{d}_1 + \wt{d}_2 \leq \frac{1}{\beta} \wt{d}_1$ and  $K_2 D_2 \leq \wt{d_1}^2 / C$ with some large $C$  by our choice of $\theta$ and $D_2$.
    This implies that $0.9$ fraction of $\wt{T}$ are unique-neighbors of $S$.

    Finally, we lower bound $|E_2(U, \wt{T})|$. Again by \Cref{eq:gadget-expansion},
    \begin{equation*}
    \begin{aligned}
        |E_2(U,\wt{T})| &\geq \sum_{v\in U_{\ell}} |S_v| \cdot 0.9 \cdot \wt{d}_1 \cdot \exp(-\theta |S_v|/\sqrt{D})
        \geq 0.9 \cdot \wt{d}_1 \cdot \exp(-\theta) \sum_{v\in U_{\ell}} |S_v| \\
        &= 2\delta \cdot \wt{d}_1 \cdot |E_1(S, U_{\ell})|
        \geq 1.6\delta d_1|S|\mcom
    \end{aligned}
    \end{equation*}
    where $\delta = \delta(\eps,\beta) > 0$.
    The last inequality uses the fact that $|E_1(S,U_\ell)| \geq 0.8 K|S|$ and $d_1 = K\wt{d}_1$.
    With $d_{\text{right}}(U, \wt{T}) \leq 1.1$, it follows that
    \begin{equation*}
        |\UN_{Z_n}(S)| \geq \delta d_1|S| \mper
    \end{equation*}



    For $S \subseteq R(Z_n)$, the analysis is completely symmetric.
    Indeed, we have $K_1 = K_2$, and we can verify that $K_1 D_1 \leq \wt{d_2}^2 / C$.
    Since $d_1 \leq d_2$, the unique-neighbor lower bound holds for all $S \subseteq V(Z_n)$ with $|S| \leq \mu|V(Z_n)|$.

    \parhead{Small set lossless expansion.}
    We now turn to the expansion of small subsets $S \subseteq L(Z_n)$.
    Let $U \coloneqq N_{G_n^{(1)}}(S) \subseteq M$ and $T \coloneqq N_{Z_n}(S) \subseteq R$.
    By assumption, $G_n^{(1)}$ has no bicycle of size at most $g_n$, thus \Cref{lem:small-set-expansion} states that $|U| \geq (1-\eps/2) K|S|$ (i.e., $S$ expands losslessly in $G_n^{(1)}$) assuming that
    \begin{equation*}
        |S| \leq \frac{1}{2(K+1)} \Paren{\frac{\eps}{2}(K-1)}^{\frac{1}{6}g_n - \frac{1}{2}} \mper
    \end{equation*}
    With our choice of $K$ and $g_n = \omega_n(1)$, it suffices that $|S| \leq \exp(g_n)$.

    Next, as each gadget on $v\in U$ expands with a factor of at least $\wt{d}_1$, we can lower bound $|E_2(U, T)|$ by $\wt{d}_1 \cdot |U|$.
    Moreover, the left average degree of the induced subgraph $G_n^{(2)}[U \cup T]$ is at least $\wt{d}_1$.
    Then, by \Cref{thm:expansion}, the right average degree of $G_n^{(2)}[U \cup T]$ is
    \begin{equation*}
        d_{\text{right}}(U, T) \leq 1 + \frac{\sqrt{(K_2-1)(D_2-1)}}{\wt{d}_1-1} \cdot O(1)
        \leq 1 + \frac{\eps}{2} \mcom
    \end{equation*}
    given that $K_2 D_2 \leq \eps^2 \wt{d}_1^2 / C$ for a large enough constant $C$.
    Thus, since $|E_2(U,T)| \geq \wt{d}_1 \cdot |U| \geq (1-\eps/2) \cdot K\wt{d}_1 \cdot |S| = (1-\eps/2) d_1|S|$,
    \begin{equation*}
        |T| \geq  \frac{1}{1+\eps/2} |E_2(U, T)| \geq \frac{1-\eps/2}{1+\eps/2} \cdot d_1 |S|
        \geq (1-\eps) d_1 |S| \mper
    \end{equation*}
    For $S \subseteq R(Z_n)$, the analysis is symmetric with $d_1, \wt{d}_1$ replaced by $d_2, \wt{d}_2$.
\end{proof}

%% file: lossless.tex
\section{One-sided lossless expanders}  \label{sec:one-sided-lossless}

In this section, we illustrate how the construction of Golowich \cite{Gol23} can be instantiated using the tripartite line product, and give a succinct proof of lossless expansion using our results on subgraphs of spectral expanders.

\begin{theorem}[One-sided lossless expanders] \label{thm:lossless-expander}
    For every $\beta\in(0,1/2]$ and $\eps \in (0,1)$, there are $K = K(\eps)$, $d_0 = d_0(\beta,\eps)$ such that for all $d_1,d_2\ge d_0$ which are multiples of $K$ and satisfy $1\ge\frac{d_1}{d_2} \ge \frac{\beta}{1-\beta}$, the following holds:
    there exists $\mu = \mu(\beta,\eps,d_1,d_2) > 0$ such that there is an infinite family of $(d_1, d_2)$-biregular graphs $(Z_n)_{n\ge 1}$ where

    \begin{equation*}
        |N_{Z_n}(S)| \geq (1-\eps) \cdot d_1 \cdot |S| \mcom \quad \forall S \subseteq L(Z_n)\mcom\ |S| \leq \mu |L(Z_n)| \mper
    \end{equation*}
\end{theorem}

\begin{proof}
    The construction of $Z_n$ is the same as \Cref{thm:2-sided-un-sse}, except this time we take $K_1 \gg K_2 = 1$ and $\theta = 1$.

    \begin{itemize}
        \item Let $K_1 = \ceil*{\frac{2^{10}}{\eps^4}}$ and $K_2 = 1$.

        \item Let $d_0 \geq C_0 K_1$ where $C_0 = C_0(\eps,\beta)$ is some large enough constant chosen later.
        

        \item Let $\wt{d}_1 \coloneqq d_1/ K_1$ and $\wt{d}_2 \coloneqq d_2$, both at least $C_0$.

        \item Let $\theta \coloneqq 1$.

        \item Let $D_1 \coloneqq (\wt{d}_1+\wt{d}_2) \cdot \wt{d}_2$ and $D_2 \coloneqq (\wt{d}_1+\wt{d}_2) \cdot \wt{d}_1$, and define $D \coloneqq D_1 + D_2$.

    \end{itemize}

    Since $d_1 \leq d_2$, we have $\wt{d}_1 \leq \wt{d}_2$.
    This choice of parameters satisfy the requirements of \Cref{lem:gadget}.
    Thus, applying \Cref{lem:gadget} with parameters $\theta = \tau = 1$, there is a $(\wt{d}_1,\wt{d}_2)$-biregular graph $H$ with $D_1$ left vertices and $D_2$ right vertices such that
    \begin{equation*}
        \min_{S\subseteq V(H): |S| \le t} \frac{|N_H(S)|}{|S|}  \geq (1-o_D(1)) \cdot \wt{d}_1 \cdot \exp(- t/\sqrt{D})
        \numberthis \label{eq:gadget-expansion-theta-1}
    \end{equation*}
    for $1 \leq t \leq \sqrt{D}$.
    We can set $C_0$ large enough (depending only on $\eps,\beta$) such that the $o_D(1)$ term is at most $\eps/8$.

    For the tripartite base graph $G_n = (L \cup M \cup R, E_1 \cup E_2)$, we construct $G_n^{(1)} = (L \cup M, E_1)$ to be a $(K_1, D_1)$-biregular near-Ramanujan graph, and $G_n^{(2)} = (M \cup R, E_2)$ is simply a $(D_2, 1)$-biregular graph.

    Next, we analyze the vertex expansion of a subset $S \subseteq L(Z_n)$ in the product graph $Z_n$.
    Similar to the proof of \Cref{thm:2-sided-un-sse}, let $U \coloneqq N_{G_n}(S) \subseteq M$ be the neighbors of $S$ in $G_n^{(1)}$, and we partition $U$ into $U_\ell \coloneqq \{v\in U: |E_1(v, S)| \leq \frac{\eps}{4} \sqrt{D} \}$ (the ``low $S$-degree'' vertices) and $U_h \coloneqq U \setminus U_\ell$ (the ``high $S$-degree'' vertices).
    Consider the bipartite subgraph induced by $S \cup U_h$.
    By definition, the right average degree in $G_n^{(1)}[S \cup U_h]$ is at least $\frac{\eps}{4}\sqrt{D}$.
    By \Cref{thm:expansion}, we can bound the left average degree by
    \begin{equation*}
        d_{\text{left}}(S,U_h) \leq 1 + \frac{\sqrt{(K_1-1)(D_1-1)}}{\frac{\eps}{4}\sqrt{D}-1} \cdot 1.1
        \leq \frac{8}{\eps}\sqrt{K_1}
    \end{equation*}
    as long as $|S| \leq \mu|L|$ for some $\mu = \mu(K_1,D_1) > 0$ (depending only on $\eps,\beta, d_1, d_2$).
    Since $K_1 \geq \frac{2^{10}}{\eps^4}$, the above is at most $\frac{\eps}{4} K_1$.
    Thus, we know that $|E_1(S, U_\ell)| \geq (1-\eps/4)K_1|S|$, i.e., most edges incident to $S$ go to $U_{\ell}$.

    Moreover, for each $v\in U_\ell$, the gadget placed on $H$ has at most $\frac{\eps}{4}\sqrt{D}$ vertices on the left, thus by \Cref{eq:gadget-expansion-theta-1} each gadget expands losslessly.
    Specifically, denoting $T \coloneqq N_{Z_n}(S) \subseteq R$, we have that
    \begin{equation*}
        |E_2(v,T)| \geq |E_1(v,S)| \cdot (1-o_D(1)) \cdot \wt{d}_1 \cdot \exp(-\eps/4)
        \geq (1-\eps/2) \cdot \wt{d}_1 \cdot |E_1(v,S)|
    \end{equation*}
    for large enough $C_0$ (hence large enough $D$).

    Finally, since $G_n^{(2)}$ is a $(D_2, 1)$-biregular graph,
    \begin{equation*}
        |T| = \sum_{v\in U} |E_2(v,T)| \geq (1-\eps/2) \cdot \wt{d}_1 \cdot |E_1(S, U_{\ell})|
        \geq (1-\eps) \cdot d_1 \cdot |S| \mper
    \end{equation*}
    This completes the proof.
\end{proof}

%% file: moore.tex
\section{Generalization of the Moore bound}
\label{app:exact-moore-bound}

In this section, we strengthen the classical Moore bound of Alon, Hoory and Linial~\cite{AHL02} and generalize the result to \emph{bicycles} (recall \Cref{def:excess}).
Our proof closely follows Section A of \cite{HKM23}, which is an alternative proof of the Moore bound.

\restatetheorem{thm:moore-bound}

The proof of \Cref{thm:moore-bound} is based on \emph{non-backtracking walks}, which are walks such that no edge is the inverse of its preceding edge.
For a graph $G$ on $n$ vertices with adjacency matrix $A$, we define $A^{(s)}$ to be the $n\times n$ matrix whose $(u, v)$ entry counts the number of length-$s$ non-backtracking walks between vertices $u$ and $v$ in $G$.
The following is a standard fact.
\begin{fact}[Recurrence and generating function of $A^{(s)}$]
\label{fact:non-backtracking}
    The non-backtracking matrices $A^{(s)}$ satisfy the following recurrence:
    \begin{equation*}
    \begin{aligned}
        A^{(0)} &= \Id \mcom \\
        A^{(1)} &= A \mcom \\
        A^{(2)} &= A^2 - D \mcom \\
        A^{(s)} &= A^{(s-1)} A - A^{(s-2)} (D-\Id) \mcom \quad s > 2 \mper
    \end{aligned}
    \end{equation*}
    The recurrences imply that these matrices have a generating function:
    \begin{equation*}
        J(t) \coloneqq \sum_{s=0}^{\infty} A^{(s)} t^{s} = (1-t^2) \cdot H(t)^{-1}
    \end{equation*}
    for $t \in [0, 1)$ whenever the series converges, where we recall that $H(t) = (D-\Id)t^2 - At + \Id$.
\end{fact}

We first state the following simple lemma from \cite{HKM23},
\begin{lemma} \label{lem:upper-bound-l-walks}
    Let $s,k\in \N$, $s\geq k$, and let $q,r$ be the quotient and remainder of $s$ divided by $k$, i.e.\ $s = qk + r$.  Then,
    \begin{equation*}
        \tr(A^{(s)}) \leq \sqrt{n} \cdot \norm{A^{(k)}}_2^{q} \cdot \norm{A^{(r)}}_F.
    \end{equation*}
\end{lemma}

With \Cref{fact:non-backtracking} and \Cref{lem:upper-bound-l-walks}, we now prove \Cref{thm:moore-bound} by analyzing the convergence of $J(t)$ as $t$ increases from $0$.

\begin{proof}[Proof of \Cref{thm:moore-bound}]
    Let $A$ be the adjacency matrix of $G$ with average degree $d>2$, let $D$ be the diagonal degree matrix $G$, and let $\rho = \lambda_1(B_G)$.
    We will analyze the convergence of $\tr(J(t)) = \sum_{s=0}^{\infty} \tr(A^{(s)}) t^s$ as $t$ increase from $0$ to $1/\rho$.
    In particular, by \Cref{lem:spectral-radius} we have that $H_G(t) \succ 0$ (thus $\tr(J(t)) < \infty$) for all $t \in [0, 1/\rho)$, and $\tr(J(1/\rho))$ diverges.

    Fix $k\in \N$. For each $s\in\N$ we can write $s = qk+r$, and
    \begin{equation*}
        J(t) = \sum_{s=0}^{\infty} A^{(s)} t^{s}
        = \sum_{r=0}^{k-1} \sum_{q=0}^{\infty} A^{(qk+r)} t^{qk+r} \mper
    \end{equation*}
    By \Cref{lem:upper-bound-l-walks}, we have
    \begin{equation*}
        \tr(J(t)) \leq \sum_{r=0}^{k-1} t^r \sqrt{n} \|A^{(r)}\|_F \sum_{q=0}^\infty \|A^{(k)}\|_2^q \cdot t^{qk}
        = \sum_{r=0}^{k-1} t^r \sqrt{n} \|A^{(r)}\|_F \sum_{q=0}^\infty \Paren{\|A^{(k)}\|_2 \cdot t^k}^q \mper
        \numberthis \label{eq:tr-J}
    \end{equation*}

    Now, let $k \coloneqq \floor{\log_{\rho} n} + 1$ and suppose for contradiction that $G$ contains no cycle of size $\leq \ell = 2k$.
    Observe that every entry of $A^{(k)}$ must be either 0 or 1, otherwise if $A^{(k)}[i,j] > 1$ then there are two distinct length-$k$ paths from $i$ to $j$, meaning there is a cycle of length at most $2k=\ell$, a contradiction.
    Therefore, the $L_1$ norm of each row of $A^{(k)}$ is at most $n$, hence $\|A^{(k)}\|_2 \leq n$.
    Then, setting $t = 1/\rho$, we have $\|A^{(k)}\|_2 \cdot (1/\rho)^k < 1$ since $k > \log_{\rho} n$, and \Cref{eq:tr-J} shows that $\tr(J(1/\rho)) < \infty$.
    This contradicts that $\tr(J(1/\rho))$ must diverge.

    Similarly, let $k' \coloneqq \floor{\log_\rho 2n}+1$ and suppose for contradiction that $G$ hs no bicycle of size $\leq \ell' = 3k'$.
    We claim that three distinct non-backtracking walks of a given length-$k'$ between any two vertices must form a bicycle, hence every entry of $A^{(k')}$ must be at most $2$.
    Suppose the union of the three distinct nonbacktracking walks between vertices $u$ and $v$, called $H_{uv}$, did not give rise to a bicycle, its excess must be at most $0$.
    Since $H_{uv}$ is connected, it must have at most one cycle.
    If there are no cycles, then there is exactly one nonbacktracking walk from $u$ to $v$, so we assume there is exactly one cycle.
    Any nonbacktracking walk in $H_{uv}$ can enter and exit the cycle at most once.
    Further, there is a unique way to start from $u$ and enter the cycle, and a unique way to exit the cycle and arrive at $v$.
    Between entering and exiting the cycle, there are only two choices: walking in the cycle clockwise or counterclockwise.
    There are at most two ways to walk between $u$ and $v$ in $k'$ steps --- either the shortest path between them is of length exactly $k'$ and does not touch the cycle, or a length-$k$ nonbacktracking walk must enter the cycle, which we established gives at most $2$ distinct walks.

    Thus, $\|A^{(k')}\|_2 \leq 2n$ and $\|A^{(k')}\|_2 \cdot (1/\rho)^{k'} < 1$ since $k' > \log_{\rho} 2n$.
    Again, \Cref{eq:tr-J} shows that $\tr(J(1/\rho)) < \infty$, a contradiction.
    This completes the proof.
\end{proof}

%% file: gadget.tex
\section{Expansion profile of random graphs}    \label{sec:gadget}

In this section we prove \Cref{lem:gadget} (existence of biregular graphs with good expansion profile).
We first prove the desired statement for \erdos--\renyi graphs given in \Cref{lem:erdos-renyi-gadget}, and then transfer the result to random regular graphs via a coupling articulated in \Cref{lem:stoc-dom}.
See \Cref{sec:prelims} for the notations of various random bipartite graph models.

\begin{lemma}   \label{lem:erdos-renyi-gadget}
    Let $\bH \sim \ER_{n_1,n_2,p}$ where $n_1 \ge n_2$.
    Then with probability $1-O\left(\frac{1}{n_1} + \frac{1}{n_2}\right)$, for all $t$:
    \[
        P_{\bH}(t) \ge p \parens*{1-p}^{t-1} n_2 - \sqrt{4 p \parens*{1-p}^{t-1} n_1 \log n_1}. 
    \]
\end{lemma}

\begin{lemma}[Embedding \erdos--\renyi graphs into random regular graphs]   \label{lem:stoc-dom}
    Fix $n_1,n_2, d_1, d_2 \in \N$ such that $m = n_1 d_1 = n_2 d_2$.
    Then for $p = \Paren{1-C \bigparen{\frac{d_1 d_2}{m} + \frac{\log m}{\min\{d_1,d_2\}}}^{1/3}} \frac{m}{n_1 n_2}$,
    there is a joint distribution of $\bH \sim \ER_{n_1,n_2,p}$ and $\bR \sim \Reg_{n_1,n_2,d_1,d_2}$ such that
    \begin{equation*}
        \Pr[\bH \subset \bR] = 1 - o(1).
    \end{equation*}
\end{lemma}

We first give a proof of \Cref{lem:gadget} assuming the \Cref{lem:erdos-renyi-gadget,lem:stoc-dom}.
\Cref{lem:erdos-renyi-gadget} is proved \hyperlink{proof:ER-gadget}{later} in this section, and \Cref{lem:stoc-dom} is proved in \Cref{sec:sandwich}.

\begin{proof}[{Proof of \Cref{lem:gadget}}]
    Recall that we would like to show that for $1 \geq \frac{d_1}{d_2} = \frac{D_2}{D_1} \geq \frac{\beta}{1-\beta}$ and $\theta\sqrt{D}/2 \leq d_1+d_2 \leq \theta\sqrt{D}$, there exists a $(d_1,d_2)$-biregular graph $R$ with $D_1$ and $D_2$ vertices on the left and right respectively such that $P_R(t)$ is large.

    By \Cref{lem:stoc-dom} there is a coupling between $\bR\sim\Reg_{D_1,D_2,d_1,d_2}$ and $\bH\sim\ER_{D_1,D_2,p}$ such that $\bH\subset \bR$ with probability $1-o_D(1)$ where $p = \parens*{1-\frac{C_{\beta}\log^{1/3}D }{D^{1/6}}} \frac{d_1}{D_2}$ where $C_{\beta}$ is a constant depending on $\beta$ and $\theta$.
    Note that $p \leq \frac{d_1}{D_2} = \frac{d_1+d_2}{D_1+D_2} \leq \frac{\theta}{\sqrt{D}}$.

    By concentration of the binomial random variable and the union bound, with probability $1-o_D(1)$ all vertices have degree $o(d_2) = o(\sqrt{D})$ in $\bR\setminus\bH$.
    Consequently: $P_{\bR}(t) \ge P_{\bH}(t) - o(\sqrt{D})$.
    Thus, it suffices to lower bound $P_{\bH}(t)$ to obtain a lower bound on $P_{\bR}(t)$.

    By \Cref{lem:erdos-renyi-gadget}, for $t \geq 1$,
    \begin{align*}
        P_{\bH}(t) &\ge p (1-p)^{t-1} D_2 - \sqrt{4p \parens*{1-p}^{t-1} D_1 \log D_1} \\
        &\ge \Paren{d_1 (1-p)^{t-1} - \sqrt{4d_2 \parens*{1-p}^{t-1} \log D } } \cdot  \parens*{1-o_D(1)}.
    \end{align*}
    For $t \le L \sqrt{D}$, $(1-p)^{t-1}$ is $\Omega(1)$ because $p \leq \frac{\theta}{\sqrt{D}}$, and since $d_2 \le \theta\sqrt{D}$ and $d_1 \ge \beta(d_1 + d_2) \ge \beta\cdot \theta\sqrt{D}/2$, we can conclude that the above is at least
    $\parens*{1-o_D(1)} \cdot d_1 \cdot (1-p)^{t-1}$.
    Finally, since $p \leq \frac{\theta}{\sqrt{D}}$, for $t \le L\sqrt{D}$,
    \[
        (1-p)^{t-1} \ge (1-o_D(1))\cdot\exp(-pt) 
        \ge (1-o_D(1)) \cdot \exp(-\theta t/\sqrt{D}),
    \]
    which completes the proof.
\end{proof}

We now prove \Cref{lem:erdos-renyi-gadget}: we show a lower bound on the expansion profile of $\ER_{n_1,n_2,p}$ using standard concentration inequalities and union bound.

\begin{proof}[{Proof of \Cref{lem:erdos-renyi-gadget}}] \hypertarget{proof:ER-gadget}{}
    Write $S\subseteq V(\bH)$, write $S\coloneqq S_L \cup S_R$ where $S_L \coloneqq S\cap L(\bH)$ and $S_R \coloneqq S \cap R(\bH)$.
    Observe that $\abs*{\UN_{\bH}(S)} = \abs*{\UN_{\bH}(S_L)} + \abs*{\UN_{\bH}(S_R)}$. Therefore, without loss of generality we can study $S$ completely in $L(\bH)$ or $R(\bH)$.

    For $S\subseteq R(\bH)$ with $|S| = t$, we have:
    \begin{align*}
        |\UN_{\bH}(S)| &= \sum_{v\in L(\bH)} \Ind[v\in\UN_{\bH}(S)].
    \end{align*}
    For each $v\in L(\bH)$, the number of edges between $v$ and $S$ is distributed as $\Bin(t, p)$,
    so each $\Ind\bracks*{v\in\UN_{\bH}(S)}$ is an independent Bernoulli with bias $q_t \coloneqq t p \parens*{1 - p}^{t-1}$. 
    By the Chernoff bound:
    \[
        \Pr\bracks*{ \abs*{\UN_{\bH}(S)} \le q_t n_1 - s\sqrt{q_tn_1}} \le \exp(-s^2 / 2),
    \]
    which in particular implies that $\abs*{\UN_{\bH}(S)} \ge q_tn_1 - \sqrt{4q_tn_1t\log n_2}$ except with probability at most $n_2^{-2t}$.
    By a union bound over all $S\subseteq R(\bH)$ of size $t$,
    \[
        \forall S\subseteq R(\bH)\text{ s.t. }|S|=t: \abs*{\UN_{\bH}(S)} \ge q_tn_1 - \sqrt{4q_tn_1t\log n_2}
    \]
    with probability at least $1-n_2^{-t}$.
    By an identical argument,
    \[
        \forall S\subseteq L(\bH)\text{ s.t. }|S|=t: \abs*{\UN_{\bH}(S)} \ge q_t n_2 - \sqrt{4q_t n_2t\log n_1}
    \]
    with probability at least $1 - n_1^{-t}$.
    In both cases, since $n_1\geq n_2,$ we have
    \[
        \frac{|\UN_{\bH}(S)|}{|S|} \geq p(1-p)^{t-1} n_2 - \sqrt{4p(1-p)^{t-1} n_1\log n_1}.
    \]
    Finally, taking a union bound over all $t\ge 1$ completes the proof.
\end{proof}

%% file: sandwich.tex
\section{Coupling \erdos--\renyi and random regular graphs}   \label{sec:sandwich}

We will closely follow \cite[Section 11.5]{FK16} (which is a special case of~\cite{DFRS17}), adapted to the case of bipartite graphs.
As before, we will write random variables in \textbf{boldface}, and see \Cref{sec:prelims} for a reminder of the notation for various random bipartite graph models.

\begin{theorem}[Embedding theorem] \label{thm:sandwich}
    Fix $n_1, n_2, d_1, d_2 \in \N$ such that $m = n_1 d_1 = n_2 d_2$.
    There is a universal constant $C$ such that if $\gamma\in (0,1)$ satisfies $\gamma  \geq C \Paren{\frac{d_1 d_2}{m} + \frac{\log m}{\min\{d_1,d_2\}}}^{1/3}$,
    then for $\wt{m} \leq \floor{(1-\gamma) m}$,
    there is a joint distribution of $\bG \sim \ER_{n_1,n_2,\wt{m}}$ and $\bR \sim \Reg_{n_1,n_2,d_1,d_2}$ such that
    \begin{equation*}
        \Pr[\bG \subset \bR] = 1 - o(1).
    \end{equation*}
    Furthermore, let $p = \frac{(1-2\gamma)m}{n_1 n_2}$.
    There is a joint distribution of $\bH \sim \ER_{n_1,n_2,p}$ and $\bR \sim \Reg_{n_1,n_2,d_1,d_2}$ such that
    \begin{equation*}
        \Pr[\bH \subset \bR] = 1 - o(1).
    \end{equation*}
\end{theorem}

To prove \Cref{thm:sandwich}, we need to introduce some more notation.
With slight abuse of notation, we write
\begin{equation*}
    \bG = (\be_1,\dots, \be_m), \quad \bR = (\boldf_1,\dots, \boldf_m)
\end{equation*}
to be random orderings of the edges.
Moreover, for $t = 1,2,\dots, m$, we define random variables $\bG_t = (\be_1,\dots,\be_t)$ and $\bR_t = (\boldf_1,\dots,\boldf_t)$.

Note that for any bipartite graph $G$ of size $t$ and any edge $e \in K_{n_1,n_2} \setminus G$, the conditional probability
\begin{equation*}
    \Pr[\be_{t+1} = e | \bG_t = G] = \frac{1}{n_1 n_2 - t}.
\end{equation*}

This motivates the following definition.

\begin{definition} \label{def:stopping-time}
    Fix $\eps \in (0,1)$.
    We define $\bA_{\eps,t}$ to be the event that for all $e \in K_{n_1,n_2} \setminus \bR_t$,
    \begin{equation*}
        \Pr[\boldf_{t+1} = e | \bR_t] \geq \frac{1-\eps}{n_1 n_2 - t}.
        \numberthis \label{eq:close-to-uniform}
    \end{equation*}
    Further, we define the stopping time
    \begin{equation*}
        \bT_{\eps} = \max \{u: \text{$\bA_{\eps,t}$ occurs for all $t \leq u$} \}.
    \end{equation*}
\end{definition}

Intuitively, suppose we sample the edges of $\bR \sim \Reg_{n_1,n_2,d_1,d_2}$ one by one.
At each time step $t \leq \bT_{\eps}$, the conditional distribution of the next edge is close to uniform, which is the case for $\ER_{n_1,n_2,m}$.
Thus, the main ingredient in the proof of \Cref{thm:sandwich} is to show that $\bT_{\eps}$ is large with high probability, i.e., for most steps, the sampling process behaves roughly like $\ER_{n_1,n_2,m}$.
The following lemma is analogous to Lemma 11.18 of \cite{FK16}.

\begin{lemma}[Large stopping time] \label{lem:large-stopping-time}
    There is a universal constant $C$ such that if $\eps\in (0,1)$ satisfies $\eps \geq C \Paren{\frac{d_1 d_2}{m} + \frac{\log m}{\min\{d_1,d_2\}}}^{1/3}$, then $\bT_{\eps} \geq (1-\eps) m$ with probability $1-o(1)$.
\end{lemma}

We will defer the proof to \Cref{sec:proof-of-large-stopping-time}.
This lemma suffices to prove \Cref{thm:sandwich}.

\begin{proof}[Proof of \Cref{thm:sandwich} by \Cref{lem:large-stopping-time}]
    Recall that $m = n_1 d_1 = n_2 d_2$.
    We will define a graph process $\bR' = (\boldf_1',\dots, \boldf_m')$ coupled with $\bG = (\be_1,\dots,\be_m)\sim \ER_{n_1,n_2,m}$ and show that (1) $\boldf'_t$ and $\boldf_t$ have the same conditional distribution, and (2) $\bR'\cap\bG$ is large and contains a random subgraph $\bG' \subset \bG$ of size $\wt{m} < m$ with high probability.
    Since a subgraph $\bG'$ is also distributed as $\ER_{n_1,n_2,\wt{m}}$, this gives a coupling between $\bG'\sim \ER_{n_1,n_2,\wt{m}}$ and $\bR' \sim \Reg_{n_1,n_2,d_1, d_2}$ such that $\bG'\subset \bR'$ with high probability.

    Recall that $\Pr[\be_{t+1} = e | \bG_t] = \frac{1}{n_1n_2 - t}$ for all $t$ and $e \in K_{n_1,n_2} \setminus \bG_t$, and we write
    \begin{equation*}
        p_{t+1}(e | \bR_t) \coloneqq \Pr\Brac{\boldf_{t+1} = e | \bR_t}
    \end{equation*}
    which is at least $\frac{1-\eps}{n_1n_2 - t}$ for $t \leq \bT_{\eps}$ by \Cref{def:stopping-time}.

    The graph process $\bR'$ is sampled as follows:
    at time step $t \leq \bT_{\eps}$,
    \begin{enumerate}
        \item Sample a Bernoulli random variable $\bxi_{t+1} \in \{0,1\}$ with bias $1-\eps$.
        \item Sample a random edge $\bg_{t+1} \in K_{n_1,n_2} \setminus \bR_t'$ according to the conditional distribution
        \begin{equation*}
            \Pr\Brac{\bg_{t+1} = e | \bR_t', \bG_t} \coloneqq \frac{1}{\eps} \Paren{ p_{t+1}(e| \bR_t') - \frac{1-\eps}{n_1 n_2 - t}} \geq 0.
        \end{equation*}
        Note that this is a valid probability distribution over $K_{n_1,n_2}\setminus \bR_t'$ since the above is non-negative due to \Cref{eq:close-to-uniform} and the sum is 1 because $\Abs{K_{n_1,n_2} \setminus \bR_t'} = n_1n_2-t$.

        \item Fix any bijection map $\bh: \bR_t' \setminus \bG_t \to \bG_t \setminus \bR_t'$.
        Set
        \begin{equation*}
            \boldf_{t+1}' =
            \begin{cases}
                \be_{t+1}, & \text{if } \bxi_{t+1} = 1,\ \be_{t+1} \notin \bR_t', \\
                \bh(\be_{t+1}), & \text{if } \bxi_{t+1} = 1,\ \be_{t+1} \in \bR_t', \\
                \bg_{t+1}, & \text{if } \bxi_{t+1} = 0.
            \end{cases}
        \end{equation*}
        Note that if $\bxi_{t+1} = 1$, then $\boldf_{t+1}' \in \bG_{t+1}$.
    \end{enumerate}
    For $t > \bT_{\eps}$, we sample $\boldf_{t+1}'$ according to the probabilities $p_{t+1}(e|\bR_t')$ without coupling, and we keep sampling $\bxi_{t+1}$ for notational convenience.

    We first show that the conditional distribution of $\boldf_{t+1}'$ is the same as $\boldf_{t+1}$.
    For $e \in K_{n_1,n_2} \setminus \bR_t'$,
    \begin{align*}
        \Pr\Brac{\boldf_{t+1}' = e | \bR_t', \bG_t}
        & = (1-\eps) \cdot \Pr\Brac{\boldf_{t+1}' = e | \bR_t', \bG_t, \bxi_{t+1} = 1} + \eps \cdot  \Pr\Brac{\bg_{t+1} = e | \bR_t', \bG_t} \\
        &= \frac{1-\eps}{n_1 n_2 - t} + \Paren{p_{t+1}(e|\bR_t') - \frac{1-\eps}{n_1n_2 - t}} \\
        &= p_{t+1}(e|\bR_t').
    \end{align*}
    This is because if $\bxi_{t+1}=1$ then $\boldf_{t+1}'$ is an edge in $\bG_{t+1}$, and if $\bxi_{t+1} = 0$ then $\boldf_{t+1}' = \bg_{t+1}$.
    This shows that $\boldf_{t+1}'$ and $\boldf_{t+1}$ have the same conditional distribution.

    Next, we claim that in the end, $\bR'$ and $\bG$ share many edges.
    Let
    \begin{equation*}
        \bS \coloneqq \Set{\boldf_{t}' : \bxi_t = 1, 0\leq t \leq (1-\eps)m} \subset \bR'.
    \end{equation*}
    By \Cref{lem:large-stopping-time}, $\bT_{\eps} \geq (1-\eps)m$ with probability $1-o(1)$.
    Conditioned on this, we know that all edges in $\bS$ lie in $\bG$.
    Moreover, $|\bS|$ is distributed as $\Bin((1-\eps)m, 1-\eps)$, so $\E[|\bS|] = (1-\eps)^2 m$, and by Chernoff bound,
    \begin{equation*}
        \Pr\Brac{|\bS| \leq (1-\eps)^3 m} \leq \exp\Paren{-\frac{1}{2}\eps^2(1-\eps)^2 m} = o(1).
    \end{equation*}
    Let $\gamma = 3\eps$ and fix $\wt{m} \leq \floor{(1-\gamma) m} \leq (1-\eps)^3 m$.
    We now take the first $\wt{m}$ edges from $\bS \subset \bR'$, and the resulting graph $\bG'$ is distributed as $\ER_{n_1,n_2, \wt{m}}$.
    Thus, we have obtained a joint distribution between $\bG' \sim \ER_{n_1,n_2,\wt{m}}$ and $\bR' \sim \Reg_{n_1,n_2,d_1,d_2}$ such that $\bG' \subset \bR'$ with probability $1-o(1)$.

    The second statement of the theorem is a simple modification.
    We sample $\ER_{n_1,n_2,p}$ as follows: (1) sample $\boldm' \sim \Bin(n_1 n_2, p)$, and (2) sample $\bH \sim \ER_{n_1,n_2, \boldm'}$ coupled with $\bR'\sim \Reg_{n_1,n_2,d_1,d_2}$ as described before (if $\boldm' > m$ then sample the extra edges randomly).
    By the Chernoff bound,
    \begin{equation*}
        \Pr[\boldm' \geq (1+\gamma) n_1 n_2 p] \leq \exp\Paren{-\gamma^2 n_1 n_2 p/3} = o(1).
    \end{equation*}
    Since $p = \frac{(1-2\gamma)m}{n_1 n_2}$, $\floor{(1+\gamma) n_1 n_2 p} \leq \floor{(1-\gamma)m}$, and conditioned on $\boldm' \leq \floor{(1-\gamma)m}$, the exact same analysis goes through.
    Thus, we get $\Pr[\bH \subset \bR'] = 1 - o(1)$, completing the proof.
\end{proof}

\subsection{Random graph extension}

To prove \Cref{lem:large-stopping-time}, we first need a few definitions and lemmas about extensions of graphs.
As before, fix $n_1, n_2, d_1, d_2 \in \N$ such that $m \coloneqq n_1 d_1 = n_2 d_2$.
We first introduce the following definitions.

\begin{definition}[Graph extension] \label{def:graph-extension}
    Given an ordered bipartite graph $G = (e_1,\dots,e_t)$, we say that an ordered simple $(d_1,d_2)$-biregular graph $H = (f_1,\dots,f_m)$ with $m$ edges is an \emph{extension} of $G$ if $e_i = f_i$ for $i\leq t$.
    We write $\calS_G \coloneqq \calS_G(n_1,n_2,d_1,d_2)$ to denote the set of extensions of $G$, and write $\bS_G$ as a random graph sampled uniformly from $\calS_G$ (we will drop the dependence on $n_1,n_2,d_1,d_2$ when clear from context).

    Given $G$ and an extension $H$, for vertices $u, v \in L$ or $u,v\in R$,
    \begin{equation*}
        \deg_{H|G}(u,v) = \Abs{ \Set{w: (u,w)\in H\setminus G \text{ and } (v,w) \in H }}.
    \end{equation*}
    Note that $\deg_{H|G}(u,v)$ is not symmetric in $u$ and $v$.
\end{definition}

Although \Cref{def:graph-extension} is a natural definition, it is difficult to analyze since we require the extension of $G$ to be \emph{simple}.
On the other hand, if we allow \emph{multigraphs} (parallel edges allowed), then there is a very simple process to sample a multigraph extension from $G$, namely the ``configuration model''.
Furthermore, it is easy to see that conditioned on the sampled multigraph being simple, the process gives the uniform distribution over $\calS_G$.

\begin{definition}[Random multigraph extension] \label{def:multigraph-extension}
    Given a graph $G = (e_1,\dots,e_t)$ of size $t$, we denote $\bM_G$ to be an ordered random \emph{multigraph} extension of $G$ sampled as follows:
    \begin{enumerate}
        \item Set $U$ to be a random permutation of $(1,\dots,1, \dots, n_1,\dots,n_1)$ where each $u\in [n_1]$ has multiplicity $d_1 - \deg_G(u)$.
        $U$ has length $n_1 d_1 - |G| = m-t$.

        \item Set $V$ to be a random permutation of $(1,\dots,1, \dots, n_2,\dots,n_2)$ where each $v\in [n_2]$ has multiplicity $d_2 - \deg_G(v)$.
        $V$ also has length $m-t$.

        \item Set the $i$-th edge of $\bM_G$ to be $e_i$ for $i\leq t$, and set the $(t+j)$-th edge of $\bM_G$ to be $(U[j], V[j])$ for each $1 \leq j \leq m-t$.
    \end{enumerate}
\end{definition}

\begin{fact} \label{fact:conditioned-on-simple}
    There are $\frac{(m-t)!}{\prod_{u\in [n_1]} (d_1 - \deg_G(u))!}$ distinct permutations of $U$ and $\frac{(m-t)!}{\prod_{v\in [n_2]} (d_2 - \deg_G(v))!}$ distinct permutations of $V$.
    Suppose $H$ is a simple (ordered) extension of $G$, then
    \begin{equation*}
        \Pr\Brac{\bM_G = H} = \frac{\prod_{u\in [n_1]} (d_1 - \deg_G(u))! \cdot \prod_{v\in [n_2]} (d_2 - \deg_G(v))!}{((m-t)!)^2}.
    \end{equation*}
    In particular, conditioned on $\bM_G$ being simple, it has the same distribution as $\bS_G$.
\end{fact}

The main ingredient of the proof of \Cref{lem:large-stopping-time} is the following lemma, which states that the probability of a random multigraph extension of $G\cup e$ being simple is roughly the same for all $e\notin G$, assuming that $G$ is not too ``saturating''.

\begin{lemma} \label{lem:simple-probability-ratio}
    Let $\eps \geq C \Paren{\frac{d_1 d_2}{m}}^{1/3} + C\sqrt{\frac{\log m}{\min\{d_1, d_2\}}}$ for a large enough constant $C$.
    Let $G$ be a bipartite graph with $t \leq (1-\eps)m$ edges such that $\calS_G = \calS_G(n_1,n_2,d_1,d_2)$ is non-empty.
    If $\deg_G(u) \leq (1-\eps/2)d_1$ and $\deg_G(v) \leq (1-\eps/2) d_2$ for all $u\in [n_1]$, $v\in [n_2]$, then for every $e,e' \notin G$, we have
    \begin{equation*}
        \frac{\Pr[\bM_{G\cup e'} \in \calS_{G\cup e'}]}{\Pr[\bM_{G\cup e} \in \calS_{G\cup e}]}
        \geq 1- \frac{\eps}{2}.
    \end{equation*}
\end{lemma}

\paragraph{Typical random extension.}

The following is analogous to Lemma 11.20 in \cite{FK16} and will be used to prove \Cref{lem:simple-probability-ratio}.
It roughly states that a random extension $\bS_G$ of a graph $G$ behaves nicely.

\begin{lemma} \label{lem:random-extension-facts}
    Let $\eps m \geq 4d_1 d_2$.
    Let $G$ be a graph with $t \leq (1-\eps)m$ edges such that $\calS_G$ is non-empty, and let $\bS_G$ be a uniform sample from $\calS_G$.
    For every $e\notin G$, we have
    \begin{equation*}
        \Pr[e \in \bS_G] \leq \frac{2d_1 d_2}{\eps m}.
    \end{equation*}
    Moreover, for every $u_1, u_2\in [n_1]$ and $\ell \geq \ell_1 \coloneqq \ceil{\frac{4d_1^2 d_2}{\eps m}}$.
    \begin{equation*}
        \Pr\Brac{\deg_{\bS_G|G}(u_1,u_2) > \ell} \leq 2^{-(\ell-\ell_1)},
    \end{equation*}
    and for every $v_1, v_2\in [n_2]$ and $\ell \geq \ell_2 \coloneqq \ceil{\frac{4d_1 d_2^2}{\eps m}}$.
    \begin{equation*}
        \Pr\Brac{\deg_{\bS_G|G}(v_1,v_2) > \ell} \leq 2^{-(\ell-\ell_2)}.
    \end{equation*}
\end{lemma}
\begin{proof}
    Fix $e = (u,v) \notin G$. We first define
    \begin{equation*}
        \calG_{\in e} = \Set{H \in \calS_G : e\in H}, \quad
        \calG_{\notin e} = \Set{H \in \calS_G : e \notin H} .
    \end{equation*}
    Then, $\Pr[e\in \bS_G] = \frac{|\calG_{\in e}|}{|\calG_{\in e}| + |\calG_{\notin e}|} \leq \frac{|\calG_{\in e}|}{|\calG_{\notin e}|}$.
    We will proceed to upper bound this ratio.

    Define a bipartite graph $B$ between $\calG_{\in e}$ and $\calG_{\notin e}$ as follows.
    We connect $H\in \calG_{\in e}$ and $H' \in \calG_{\notin e}$ if we can obtain $H'$ from $H$ with the following switching operation:
    choose an edge $(w,x) \in H\setminus G$ disjoint from $(u,v)$ such that $(u,x)$ and $(w,v)$ are not edges in $H$, and replace $(u,v)$, $(w,x)$ by $(u,x)$, $(w,v)$.
    We write $\deg_B(H)$ to denote the degree of $H$ in $B$.

    For $H \in \calG_{\in e}$, we can choose edge $(w,x)\in H\setminus G$ as long as $w$ is not a neighbor of $v$ and $x$ is not a neighbor of $u$ in $H$.
    Thus, there are at least $|H| - |G| - \deg_H(u) \cdot \deg_H(v) = m-t- d_1d_2$ choices, meaning $\deg_B(H) \geq m-t-d_1 d_2$.

    On the other hand, for $H' \in \calG_{\notin e}$, we must select $x$ to be a neighbor of $u$ and $w$ a neighbor of $v$ in $H'$.
    Thus, $\deg_B(H') \leq d_1 d_2$.

    Since $B$ is a bipartite graph, we must have
    \begin{equation*}
        |\calG_{\in e}| \cdot \min_{H\in \calG_{\in e}} \deg_B(H) \leq |\calG_{\notin e}| \cdot \max_{H'\in \calG_{\notin e}} \deg_B(H')
        \implies
        \frac{|\calG_{\in e}|}{|\calG_{\notin e}|} \leq \frac{d_1 d_2}{m-t-d_1 d_2}
        \leq \frac{2d_1 d_2}{\eps m},
    \end{equation*}
    since $t \leq (1-\eps)m$ and $\eps m \geq 4d_1 d_2$.

    To prove the second statement, for $u_1,u_2 \in [n_1]$, we define
    \begin{equation*}
        \calG_{\ell} = \Set{H \in \calS_G: \deg_{H|G}(u_1,u_2) = \ell}
    \end{equation*}
    for $\ell = 0,1,\dots$, and we adopt a similar strategy of constructing an auxiliary bipartite graph $B_{\ell}$ between $\calG_{\ell}$ and $\calG_{\ell-1}$.
    We connect $H\in \calG_{\ell}$ to $H'\in \calG_{\ell-1}$ if we can obtain $H'$ from $H$ by the following:
    (1) select a vertex $w$ contributing to $\deg_{H|G}(u_1,u_2)$, i.e., $(u_1,w) \in H\setminus G$ and $(u_2,w) \in H$,
    (2) select a disjoint edge $(u',w') \in H\setminus G$ such that there is no edge between $\{u',w'\}$ and $\{u_1,u_2,w\}$ in $H$,
    and (3) replace edges $(u_1, w)$ and $(u',w')$ with $(u',w)$ and $(u, w')$.

    By a similar analysis, fix $H\in \calG_{\ell}$, there are $\ell$ choices for $w$ and at least $m-t - 2d_1 d_2$ choices for $(u',w')$, thus $\deg_{B_\ell}(H) \geq \ell(m-t- 2 d_1 d_2) \geq \frac{1}{2}\ell \eps m$.
    On the other hand, fix $H' \in \calG_{\ell-1}$, there are at most $d_1$ choices for $w$, $d_1$ choices for $u'$, and $d_2$ choices for $w'$, thus $\deg_{B_{\ell}}(H') \leq d_1^2 d_2$.
    Therefore,
    \begin{equation*}
        \frac{|\calG_{\ell}|}{|\calG_{\ell-1}|}
        \leq \frac{2d_1^2 d_2}{\ell \eps m}
        \leq \frac{1}{2}
    \end{equation*}
    for all $\ell \geq \ell_1 \coloneqq \ceil{\frac{4d_1^2 d_2}{\eps m}}$.
    Therefore,
    \begin{equation*}
        \Pr\Brac{\deg_{\bS_G|G}(u_1,u_2) > \ell} = \frac{\sum_{k > \ell} |\calG_k|}{\sum_{k\geq 0} |\calG_k|}
        \leq \frac{\sum_{k>\ell} |\calG_k|}{|\calG_{\ell_1}|}
        \leq \sum_{k>\ell} 2^{-(k-\ell_1)}
        \leq 2^{-(\ell-\ell_1)}.
    \end{equation*}

    For $v_1,v_2\in [n_2]$, the same analysis shows that $\Pr\brac{\deg_{\bS_G|G}(v_1,v_2) > \ell} \leq 2^{-(\ell-\ell_2)}$ for all $\ell \geq \ell_2 \coloneqq \ceil{\frac{4d_1d_2^2}{\eps m}}$.
    This completes the proof.
\end{proof}

We can now complete the proof of \Cref{lem:simple-probability-ratio}.

\begin{proof}[Proof of \Cref{lem:simple-probability-ratio}]
    We will write $\bM = \bM_{G\cup e}$ and $\bM' = \bM_{G\cup e'}$ for simplicity.
    We will construct a coupling of $\bM$ and $\bM'$ such that they differ in at most 3 positions.
    Given $\bM$ and $e = (u,v)$, $e' = (u',v')$, we perform a switching operation:
    \begin{enumerate}
        \item Delete $e$ and add $e'$ to $\bM$.
        \item Recall $U$ and $V$ of length $m-(t+1)$ defined in \Cref{def:multigraph-extension}.
        \begin{itemize}
            \item If $e$ and $e'$ are disjoint, then randomly select a copy of $u'$ in $U$ and change to $u$, and similarly randomly select a copy of $v'$ in $V$ and change to $v$.

            \item If $e$ and $e'$ are not disjoint (w.l.o.g.\ assume $v=v'$), then just change a random copy of $u'$ in $U$ to $u$.
        \end{itemize}
        Then, connect edges according to $U$ and $V$ as in \Cref{def:multigraph-extension}.
    \end{enumerate}
    Note that step 2 is equivalent to sampling a random edge $(u',w_R)$ incident to $u'$ in $\bM \setminus (G \cup e)$ and replacing $(u',w_R)$ with $(u,w_R)$, and similarly replacing a random $(w_L, v')$ with $(w_L, v)$.

    We denote the resulting graph as $\bM^*$.
    It is clear that the resulting vector $V$ after step 2 is distributed as a random permutation of $(1,\dots,1, \dots, n_2,\dots,n_2)$ with multiplicity $d_2 - \deg_{G\cup e'}(v)$ for each $v$, thus $\bM^*$ has the same distribution as $\bM'$.

    We now analyze the probability of $\bM^*$ being simple conditioned on $\bM$ being simple.
    We first identify some nice properties of $\bM$, which we will show to occur with high probability.
    We define
    \begin{equation*}
        \calG_{\nice} = \Set{H \in \calS_{G\cup e}: e' \notin H,\ \deg_{H|G\cup e}(u',u) \leq \ell_1 + \log m \text{ and } \deg_{H|G\cup e}(v',v) \leq \ell_2 + \log m},
    \end{equation*}
    where $\ell_1 = \ceil{\frac{4d_1^2 d_2}{\eps m}}$ and $\ell_2 = \ceil{\frac{4d_1 d_2^2}{\eps m}}$ as defined in \Cref{lem:random-extension-facts}.


    First, by \Cref{fact:conditioned-on-simple} and \Cref{lem:random-extension-facts}, $\Pr[e' \in \bM | \bM \in \calS_{G\cup e}] = \Pr[e' \in \bS_{G\cup e}] \leq \frac{2d_1 d_2}{\eps m} \leq \eps/8$ due to our lower bound on $\eps$.
    Moreover, we have $\deg_{\bM|G\cup e}(u',u) > \ell_1 + \log m$ and $\deg_{\bM|G\cup e}(v',v) > \ell_2 + \log m$ with probability $\leq \frac{1}{m}$.
    Thus,
    \begin{equation*}
        \Pr\Brac{\bM \in \calG_{\nice} | \bM \in \calS_{G\cup e}} \geq 1 - \frac{\eps}{4}.
    \end{equation*}

    Now suppose $\bM \in \calG_{\nice}$.
    A parallel edge can occur in three ways: (1) if we replace $(u',w_R)\in \bM \setminus (G \cup e)$ with $(u,w_R)$ but $(u,w_R)\in \bM$ already, (2) similarly for $v'$, (3) if we select $(u',v)$ and $(u, v')$ (i.e., $w_R = v$ and $w_L = u$) resulting in $(u,v)$ being parallel.
    By the union bound over these 3 cases,
    \begin{equation*}
        \Pr\Brac{\bM^* \notin \calS_{G\cup e'} | \bM \in \calS_{G\cup e} }
        \leq \frac{\deg_{\bM|G\cup e}(u', u)}{ \deg_{\bM \setminus (G\cup e)}(u') }
        + \frac{\deg_{\bM|G\cup e}(v', v)}{ \deg_{\bM \setminus (G\cup e)}(v') }
        + \frac{1}{\deg_{\bM \setminus (G\cup e)}(u') \cdot \deg_{\bM \setminus (G\cup e)}(v')}.
    \end{equation*}
    By the assumption on the degrees of $G$, we know that $\deg_{\bM\setminus (G\cup e)}(u') \geq \eps d_1 / 2$ and $\deg_{\bM\setminus (G\cup e)}(v') \geq \eps d_2 / 2$.
    \begin{align*}
        \Pr\Brac{\bM^* \notin \calS_{G\cup e'} | \bM \in \calS_{G\cup e} }
        &\leq \frac{\ell_1 + \log m}{\eps d_1/2} + \frac{\ell_2 + \log m}{\eps d_2/2} + \frac{1}{\eps^2 d_1 d_2 /4} \\
        &\leq \frac{8d_1 d_2}{\eps^2 m} + \frac{2\log m}{\eps}\Paren{\frac{1}{d_1}+ \frac{1}{d_2}} + \frac{4}{\eps^2 d_1 d_2}.
    \end{align*}
    For $\eps \geq C \paren{\frac{d_1 d_2}{m}}^{1/3} + C\sqrt{\frac{\log m}{\min\{d_1,d_2\}}}$ for some large enough constant $C$, the above can be bounded by $\eps/4$.

    Finally, we can finish the proof. As $\bM^*$ is distributed as $\bM'$,
    \begin{align*}
        \frac{\Pr[\bM' \in \calS_{G\cup e'}]}{\Pr[\bM \in \calS_{G\cup e}]}
        &\geq \frac{\Pr[\bM \in \calG_{\nice},\ \bM \in \calS_{G\cup e}]}{\Pr[\bM \in \calS_{G\cup e}]} \cdot \frac{\Pr[\bM^* \in \calS_{G\cup e'}]}{\Pr[\bM \in \calG_{\nice}]} \\
        &\geq \Paren{1- \frac{\eps}{4}}^2 \geq 1-\frac{\eps}{2},
    \end{align*}
    completing the proof.
\end{proof}

\subsection{Proof of \texorpdfstring{\Cref{lem:large-stopping-time}}{Lemma~\ref{lem:large-stopping-time}}}
\label{sec:proof-of-large-stopping-time}


\paragraph{Degree bounds.}
We first prove a degree concentration bound.
To do so, we will need concentration results for random processes without replacement.
\begin{fact}[\cite{Hoe63}] \label{fact:with-without-replacement}
    Fix $1 \leq n \leq N$.
    Let $\Omega = (x_1,\dots,x_N)$ be a finite set of points in $\R$.
    Let $\bX_1,\dots,\bX_n$ be a random sample \emph{without} replacement from $\Omega$, let $\bY_1,\dots,\bY_n$ be a random sample \emph{with} replacement from $\Omega$, and let $\bX = \sum_{i=1}^n \bX_i$ and $\bY = \sum_{i=1}^n \bY_i$.
    Suppose $f: \R\to \R$ is continuous and convex, then
    \begin{equation*}
        \E[f(\bX)] \leq \E[f(\bY)].
    \end{equation*}
\end{fact}

\Cref{fact:with-without-replacement} implies that many concentration results known for sampling \emph{with} replacement, such as the Chernoff bound, can be transferred to the case of sampling \emph{without} replacement.
In particular, the following is a standard result for which we include the proof for completeness.

\begin{lemma}[Concentration for sampling without replacement]
\label{lem:concentraiton-without-replacement}
    Fix $1 \leq k,n \leq N$.
    Let $S\subseteq [N]$, let $\bT$ be a random sample of $n$ elements from $[N]$ without replacement, and let $p = \frac{n}{N}$.
    Then, for all $\delta \in (0,1)$,
    \begin{align*}
        &\Pr\Brac{ \bigabs{|S \cap \bT| - p|S|} \geq \delta p |S|} \leq 2 \exp \Paren{-\frac{1}{3}\delta^2 p|S|}, \\
        &\Pr\Brac{ \bigabs{|S \cap \bT| - p|S|} \geq \delta (1-p) |S|} \leq 2 \exp \Paren{-\frac{1}{3}\delta^2 (1-p)|S|}.
    \end{align*}
\end{lemma}
\begin{proof}
    Let $\Omega = (x_1,\dots,x_N)$ where $x_i = 1$ if $i \in S$ and 0 otherwise.
    Let $\bX_1,\dots,\bX_n$ be a random sample without replacement from $\Omega$.
    Then, $\bX = \sum_{i=1}^n \bX_i$ has the same distribution as $|S \cap \bT|$.

    By \Cref{fact:with-without-replacement}, we can apply the Chernoff bound as if the $\bX_i$'s are sampled with replacement.
    Let $\bY_1,\dots,\bY_n$ be samples from $\Omega$ with replacement, then $\bY_i$ are i.i.d.\ Bernoulli random variables with $\E[\bY_i] = \frac{|S|}{N}$, and $\E \sum_{i=1}^n \bY_i = \frac{n|S|}{N} = p|S|$.
    The first inequality then follows from the Chernoff bound.

    For the second inequality, we look at $|S \setminus \bT|$.
    Due to symmetry, sampling $n$ elements from $[N]$ without replacement is equivalent to sampling $N-n$ elements and taking the complement.
    Let $\bX_1',\dots,\bX_{N-n}'$ be a random sample without replacement from $\Omega$.
    Then, $|S\setminus \bT|$ is distributed as $\bX' = \sum_{i=1}^{N-n} \bX_i'$, and $\E[\bX'] = (N-n) \frac{|S|}{N} = (1-p)|S|$.
    By \Cref{fact:with-without-replacement} (transferring to sampling with replacement) and the Chernoff bound,
    \begin{equation*}
        \Pr\Brac{ \bigabs{|S \setminus \bT| - (1-p)|S|} \geq \delta (1-p) |S|} \leq 2 \exp \Paren{-\frac{1}{3}\delta^2 (1-p)|S|}.
    \end{equation*}
    But $|S \cap \bT| + |S \setminus \bT| = |S|$, so $\bigabs{|S \setminus \bT| - (1-p)|S|} = \bigabs{|S \cap \bT| - p|S|}$.
    This completes the proof.
\end{proof}

Recall that $\bR = (\boldf_1,\dots,\boldf_m) \sim \Reg_{n_1,n_2,d_1,d_2}$ is a (ordered) random $(d_1,d_2)$-biregular graph on $L= [n_1]$, $R = [n_2]$,
and we write $\bR_t = (\boldf_1,\dots,\boldf_t)$.

\begin{lemma}[Degree concentration] \label{lem:degree-concentration}
    Consider the random graph process $\bR_t = (\boldf_1,\dots,\boldf_t)$ for $t=0,1,\dots, m$.
    Let $\eps \in (0,1)$.
    With probability $1- O(m^{-1})$, for all $t \leq (1-\eps)m$, letting $p = \frac{t}{m}$, we have
    \begin{align*}
        \Abs{\deg_{\bR_t}(u) - p d_1} &\leq 3\sqrt{(1-p)d_1\log m}, \quad \forall u \in [n_1], \\
        \Abs{\deg_{\bR_t}(v) - p d_2} &\leq 3\sqrt{(1-p)d_2\log m}, \quad \forall v \in [n_2].
    \end{align*}
\end{lemma}
\begin{proof}
    Fix a time step $1 \leq t \leq (1-\eps)m$ and a vertex $u\in [n_1]$.
    Since $\bR_t$ can be viewed as sampling $t$ edges from $m$ total edges without replacement, we can apply the second inequality in \Cref{lem:concentraiton-without-replacement} with $|S| = d_1$ and $\delta = 3\sqrt{\frac{\log m}{(1-p) d_1}}$ (note that $1-p \geq \eps > 0$):
    \begin{equation*}
        \Pr\Brac{ \Abs{\deg_{\bR_t}(u) - p d_1} \geq 3\sqrt{(1-p)d_1\log m}} \leq 2m^{-3}.
    \end{equation*}
    By the same analysis, the inequality is true for $v\in [n_2]$ if we replace $d_1$ with $d_2$.

    The lemma now follows from taking the union bound over all $t \leq (1-\eps)m$ and $u\in [n_1]$, $v\in[n_2]$.
\end{proof}

We are now ready to prove \Cref{lem:large-stopping-time}.

\begin{proof}[Proof of \Cref{lem:large-stopping-time}]
    We would like to prove that for all $e\in K_{n_1,n_2} \setminus \bR_t$, $\Pr[\boldf_{t+1} = e | \bR_t] \geq \frac{1-\eps}{n_1 n_2 - t}$ for every $t \leq (1-\eps)m$.
    It suffices to prove that for every $e, e' \in K_{n_1,n_2}\setminus \bR_t$,
    \begin{equation*}
        \frac{\Pr[\boldf_{t+1} = e' | \bR_t]}{\Pr[\boldf_{t+1} = e | \bR_t]} \geq 1-\eps,
    \end{equation*}
    since the average $\Pr[\boldf_{t+1} = e | \bR_t]$ over all $e$ must be $\frac{1}{n_1 n_2 - t}$, hence $\max_e \Pr[\boldf_{t+1} = e | \bR_t] \geq \frac{1}{n_1n_2 - t}$.

    Recalling the definition of extensions in \Cref{def:graph-extension}, we have
    \begin{equation*}
        \frac{\Pr[\boldf_{t+1} = e'| \bR_t]}{\Pr[\boldf_{t+1} = e | \bR_t]}
        = \frac{|\calS_{\bR_t \cup e'}|}{|\calS_{\bR_t\cup e}|}.
        \numberthis \label{eq:ratio}
    \end{equation*}
    We now consider multigraph extensions of $\bR_t \cup e$ and $\bR_t \cup e'$.
    By \Cref{fact:conditioned-on-simple},
    \begin{align*}
        \Pr\Brac{\bM_{\bR_t\cup e} \in \calS_{\bR_t\cup e}} &= |\calS_{\bR_t\cup e}| \cdot \frac{\prod_{u\in [n_1]} (d_1 - \deg_{\bR_t\cup e}(u))! \cdot \prod_{v\in [n_2]} (d_2 - \deg_{\bR_t\cup e}(v))!}{((m-t)!)^2}, \\ 
        \Pr\Brac{\bM_{\bR_t\cup e'} \in \calS_{\bR_t\cup e'}} &= |\calS_{\bR_t\cup e'}| \cdot \frac{\prod_{u\in [n_1]} (d_1 - \deg_{\bR_t\cup e'}(u))! \cdot \prod_{v\in [n_2]} (d_2 - \deg_{\bR_t\cup e'}(v))!}{((m-t)!)^2}.
    \end{align*}
    Thus, let $e = (u,v)$ and $e' = (u',v')$,
    \begin{equation*}
        \frac{\Pr[\bM_{\bR_t \cup e'} \in \calS_{\bR_t \cup e'}]}{\Pr[\bM_{\bR_t \cup e} \in \calS_{\bR_t \cup e}]}
        = \frac{|\calS_{\bR_t\cup e'}|}{|\calS_{\bR_t\cup e}|} \cdot \frac{(d_1 - \deg_{\bR_t}(u)) (d_2 - \deg_{\bR_t}(v))}{(d_1 - \deg_{\bR_t}(u')) (d_2 - \deg_{\bR_t}(v'))}.
        \numberthis \label{eq:simple-ratio}
    \end{equation*}
    Let $p = \frac{t}{m}$.
    By the concentration of degrees (\Cref{lem:degree-concentration}), with probability $1-O(m^{-1})$, for $t \leq (1-\eps)m$ (hence $1-p \geq \eps$),
    \begin{equation*}
        \frac{d_1 - \deg_{\bR_t}(u)}{(1-p) d_1} \in 1 \pm 3\sqrt{\frac{\log m}{(1-p)d_1}}
        \in 1 \pm 3 \sqrt{\frac{\log m}{\eps d_1}}.
    \end{equation*}
    Our assumption on $\eps$ implies that $\eps \geq C(\frac{\log m}{\min\{d_1, d_2\}})^{1/3}$ for a large enough $C$, so the above is bounded by $1 \pm \eps/8$.
    The same also holds for $d_2 - \deg_{\bR_t}(v)$ and $d_2 - \deg_{\bR_t}(v')$.
    Thus, we have
    \begin{equation*}
        (\ref{eq:simple-ratio}) \leq \frac{|\calS_{\bR_t\cup e'}|}{|\calS_{\bR_t\cup e}|} (1 + \eps/2).
    \end{equation*}
    On the other hand, the degree bounds also allow us to apply \Cref{lem:simple-probability-ratio}:
    \begin{equation*}
        (\ref{eq:simple-ratio}) \geq 1 - \frac{\eps}{2}.
    \end{equation*}
    Therefore, by \Cref{eq:ratio} we have $\frac{\Pr[\boldf_{t+1} = e' | \bR_t]}{\Pr[\boldf_{t+1} = e | \bR_t]} \geq 1-\eps$ with probability $1-O(m^{-1})$, completing the proof.
\end{proof}